\documentclass[AMA,STIX1COL]{WileyNJD-v2}
\usepackage{moreverb}
\newcommand\BibTeX{{\rmfamily B\kern-.05em \textsc{i\kern-.025em b}\kern-.08em
T\kern-.1667em\lower.7ex\hbox{E}\kern-.125emX}}
\articletype{Research Article}
\received{<day> <Month>, <year>}
\revised{<day> <Month>, <year>}
\accepted{<day> <Month>, <year>}

\usepackage{hyperref}
\usepackage{comment}
\newtheorem{example}{Example}
\usepackage{subfigure}
\usepackage{subcaption} 
\usepackage{float}      
\usepackage{placeins}

\newcommand{\mb}{\mathbb}

\usepackage{tikz,xcolor}
\definecolor{lime}{HTML}{A6CE39}
\definecolor{lightblue}{rgb}{0.0, 0.0, 0.5}
\DeclareRobustCommand{\orcidicon}{%
	\begin{tikzpicture}
	\draw[lime, fill=lime] (0,0)
	circle [radius=0.16]
	node[white] {{\fontfamily{qag}\selectfont \tiny ID}};
	\draw[white, fill=white] (-0.0625,0.095)
	circle [radius=0.007];
	\end{tikzpicture}
	\hspace{-2mm}
}
\foreach \x in {A, ..., Z}{%
	\expandafter\xdef\csname orcid\x\endcsname{\noexpand\href{https://orcid.org/\csname orcidauthor\x\endcsname}{\noexpand\orcidicon}}
}

\begin{document}

\title{
A Family of Iterative Methods for Computing Generalized Inverses of Quaternion Matrices and its Applications
}

\author[1]{Biswarup Karmakar}

\author[2]{Neha Bhadala}

\author[3]{Ratikanta Behera }

\authormark{Biswarup Karmakar \textsc{et al}}

\address[]{\orgdiv{Department of Computational and Data Sciences}, \orgname{Indian Institute of Science}, \orgaddress{\city{Bangalore}, \state{Karnataka}, \postcode{560012}, \country{India}}}

\corres{Ratikanta Behera, Department of Computational and Data Sciences, Indian Institute of Science, Bangalore, India. \email{ratikanta@iisc.ac.in}}

\abstract[Abstract]{
The computation of generalized inverses of quaternion matrices is a fundamental problem in quaternion linear algebra, with wide-ranging applications in signal processing, image restoration, and multidimensional data analysis. This paper presents three efficient quaternion iterative algorithms for computing the Moore--Penrose pseudoinverse: (i) the quaternion rapid iterative method (QRAPID), (ii) the quaternion strong approximate inverse (QSAI), and (iii) the quaternion hyperpower iterative method of order nineteen (QHPI$19$). Convergence theorems and perturbation bounds are established to ensure numerical stability and robustness. The QSAI method is further employed as a preconditioner for quaternion Krylov subspace solvers, resulting in substantial reductions in the iteration count and runtime for large-scale linear systems. Comprehensive numerical experiments demonstrate that the proposed algorithms achieve an accuracy comparable to or better than existing approaches—including quaternion SVD, quaternion Newton--Schulz, and classical hyperpower schemes—while offering significant computational savings. The practical utility of the framework is illustrated through two representative applications: image completion via CUR decomposition and signal filtering, which confirm its scalability and effectiveness in real-world multidimensional data applications. 
}


\maketitle
\section{Introduction}\label{sec1}
Quaternions~\cite{Hamilton2009}, introduced by Hamilton in $1843$, extend complex numbers by adding three imaginary units to the real part. 
Their noncommutative multiplication enriches the algebraic structure while introducing new analytical and computational challenges. Because of their compact and numerically stable representation of three-dimensional rotations, quaternions have become indispensable in mathematics, physics, and engineering. They have been extensively used in computer graphics, robotics, and aerospace engineering~\cite{wie1989quaternion,forbes2015fundamentals,Vince2021}, as well as in emerging areas such as machine learning~\cite{zhou2023quaternion,bill2023comparison}, signal processing~\cite{miron2023quaternions,diao2025optimizing}, and scientific computing~\cite{zhang2023partial}.  

Beyond these applications, quaternion linear algebra~\cite{Zhang1997} has emerged as a vibrant area of theoretical and computational research.  
Classical matrix concepts—including eigenvalue decompositions~\cite{farenick2003thespectral}, QR and singular value decompositions~\cite{bunse1989quaternionqr,lebihan2004svd}, and randomized low-rank approximations~\cite{liu2022randomizedsvd,ren2022randomizedqlp} have been successfully extended from real and complex matrices to the quaternion setting.  
Among these developments, the Moore--Penrose pseudoinverse occupies a particularly central role, providing least-squares solutions for overdetermined systems and minimum-norm solutions for underdetermined ones \cite{ahmadi2017iterative,yuan2016lstructured,petkovic2011iterative}.  

The most common approach for computing the Moore--Penrose pseudoinverse is based on quaternion singular value decomposition (QSVD)~\cite{wei2018quaternion}.  
While QSVD produces numerically stable and exact results, it is computationally expensive and memory-intensive, making it impractical for large or dense quaternion matrices. Alternative decomposition-based techniques, such as the quaternion full-rank decomposition~\cite{MR4509094}, still require expensive matrix factorizations and therefore remain computationally demanding and poorly scalable. In a related line of work, Song et al.~\cite{song2011cramer} derived determinantal representations for generalized inverses using quaternionic column and row determinants, but did not propose implementable algorithms.  
Huang et al.~\cite{huang2015themoore} extended the Moore--Penrose inverse to quaternion polynomial matrices through an interpolation-based technique that improved upon the classical Leverrier--Faddeev algorithm.  
More recently, Bhadala et al.~\cite{bhadala2025generalized} proposed a direct computational framework for quaternion outer inverses with prescribed range and null-space constraints, unifying several classical inverses such as the Moore--Penrose, group, and Drazin inverses.  
Despite these advances, the efficient and scalable computation of quaternion pseudoinverses for large-scale systems remains a challenging problem.

To address these limitations, iterative algorithms for computing quaternion generalized inverses have gained attention.  
Such methods avoid explicit matrix decompositions and rely primarily on quaternion matrix multiplications and adjoints, making them computationally attractive for large or structured matrices. One of the most widely used iterative schemes in the real domain is the Newton--Schulz iteration~\cite{benisrael2003generalised}, which updates an approximation $X_j$ to the pseudoinverse of a matrix \(A\) according to
\[
X_{j+1} = X_j(2I - A X_j), 
\quad j = 0,1,2,\ldots,
\]
where \(I\) denotes the identity matrix of appropriate dimension. This method achieves quadratic convergence under suitable spectral conditions. Leplat et al.~\cite{leplat2025iterative} extended this idea to quaternion matrices by introducing the hyperpower iteration
\[
X_{j+1} 
= X_j \sum_{s=0}^{p-1} (I - A X_j)^{s},
\quad j = 0,1,2,\ldots,
\]
where the parameter \(p \ge 2\) controls the degree of the residual polynomial and hence the order of convergence. 
Larger choices of \(p\) yield higher-order convergence, but require more quaternion matrix--matrix multiplications per iteration.

Recent studies in the real domain~\cite{khosravi2023gibs,cordero2021ageneralclass,soleymani2014afastconv} have proposed higher-order hyperpower-based schemes with reduced computational effort by optimizing the residual polynomial structure.  
However, such developments have not yet been explored in the quaternion framework, motivating the design of new iterative methods that preserve the high-order convergence of hyperpower schemes while improving computational efficiency and numerical stability.

In this paper, we propose three efficient quaternion iterative methods for computing the Moore--Penrose pseudoinverse: (i) the quaternion rapid iterative method (QRAPID); (ii) the quaternion strong approximate inverse (QSAI); (iii) the quaternion hyperpower iteration of order $19$ (QHPI19). The convergence theorems and perturbation bounds for all three algorithms are established, ensuring numerical stability and robustness.  Extensive numerical experiments demonstrate that the proposed methods outperform existing techniques—including QSVD, quaternion Newton--Schulz (QNS), and classical hyperpower iterations, in terms of runtime while maintaining comparable or superior accuracy.  
We further illustrate their practical utility in two representative applications:  
(i) image completion using CUR decomposition and  
(ii) filtering of chaotic three-dimensional signals.  
These results confirm the scalability, efficiency, and versatility of the proposed framework for large-scale quaternion computations.

The main contributions of this paper are summarized as follows:
\begin{itemize}
    \item We develop three new quaternion iterative algorithms—QRAPID, QSAI, and QHPI$19$—for efficient computation of the quaternion Moore--Penrose pseudoinverse.
    \item We establish rigorous convergence theorems and perturbation bounds that guaranty the stability and robustness of the proposed methods.
    \item We demonstrate the effectiveness of the proposed QSAI method as a preconditioner for quaternion Krylov subspace solvers, achieving substantial reductions in iteration counts and computation time.
    \item We validate the proposed methods on large-scale quaternion problems, including CUR-based image completion and signal filtering, highlighting their computational efficiency and practical utility.
    \item Through extensive numerical comparisons, we show that the proposed algorithms achieve an accuracy comparable to or better than existing methods while offering significant reductions in computational cost.
\end{itemize}

The remainder of this paper is organized as follows.  
Section~\ref{sec2} outlines the necessary preliminaries on quaternion algebra.  
Section~\ref{sec3} develops the general quaternion hyperpower iterative framework, presents the proposed QRAPID, QSAI, and QHPI$19$ algorithms, and analyzes their convergence properties as well as their application as preconditioners in quaternion linear systems.  
Section~\ref{sec4} provides numerical experiments and comparative performance analyzes.  
Section~\ref{sec5} demonstrates the effectiveness of the proposed methods through two representative applications: image completion and signal filtering.  
Finally, Section~\ref{sec6} concludes the paper.

\section{Preliminaries}\label{sec2}
 A quaternion $q \in \mathbb{Q}$ can be expressed as
$q = a^{(s)} + a^{(x)} \mathbf{i} + a^{(y)} \mathbf{j} + a^{(z)} \mathbf{k},$
where $a^{(s)}, a^{(x)}, a^{(y)}, a^{(z)} \in \mathbb{R}$, and 
$\mathbf{i}, \mathbf{j}, \mathbf{k}$ are the imaginary units that satisfy $\mathbf{i}^2 = \mathbf{j}^2 = \mathbf{k}^2 = -1$, and $\mathbf{i}\mathbf{j}\mathbf{k} = -1$.
The real and imaginary parts of $q$ are denoted, respectively, by
$\Re(q) = a^{(s)}$ and 
$\Im(q) = a^{(x)}\mathbf{i} + a^{(y)}\mathbf{j} + a^{(z)}\mathbf{k}$. Every quaternion $q$ can be represented in an equivalent $2 \times 2$ complex matrix form
\[
q \;\; \longleftrightarrow \;\;
\begin{bmatrix}
a^{(s)} + a^{(x)} \mathbf{i} &~ a^{(y)} + a^{(z)} \mathbf{i} \\
- a^{(y)} + a^{(z)} \mathbf{i} &~ a^{(s)} - a^{(x)} \mathbf{i}
\end{bmatrix}
\in \mathbb{C}^{2 \times 2},
\]
which allows quaternionic operations to be analyzed through their complex matrix representation. Extending this concept to matrices, any quaternion matrix
$A= A^{(s)} + A^{(x)} \mathbf{i} + A^{(y)} \mathbf{j} + A^{(z)} \mathbf{k} \in \mathbb{Q}^{m \times n}$, with components
$A^{(s)}, A^{(x)}, A^{(y)}, A^{(z)} \in \mathbb{R}^{m \times n}$, 
can be mapped into a complex matrix of size $2m \times 2n$ known as its complex representation~\cite{MR4509094}
\[
A^{C} =
\begin{bmatrix}
A^{(s)} + A^{(x)} \mathbf{i} & A^{(y)} + A^{(z)} \mathbf{i} \\
- A^{(y)} + A^{(z)} \mathbf{i} & A^{(s)} - A^{(x)} \mathbf{i}
\end{bmatrix}
\in \mathbb{C}^{2m \times 2n}.
\]
For a quaternion matrix $A \in \mathbb{Q}^{m \times n}$, the conjugate transpose (or adjoint) is defined as $A^{H} = \overline{A}^{\top}$, where $\overline{A}$ denotes elementwise quaternionic conjugation.  
A matrix is termed Hermitian if $A = A^{H}$ and unitary if $A^{H}A = I$.  
The inner product on $\mathbb{Q}^{n}$ is defined by $\langle x, y \rangle = x^{H}y$, which is conjugate–linear in the first argument and linear in the second.  
The corresponding vector norm is $\|x\|_{2} = \sqrt{\Re\langle x, x \rangle}$, while the induced operator norm of a matrix $A$ is given by~\cite{wei2018quaternion}
\[
\|A\|_{2} = \sup_{\|x\|_{2}=1} \|Ax\|_{2},
\]
which satisfies the submultiplicative property $\|AB\|_{2} \le \|A\|_{2}\|B\|_{2}$. Throughout this paper, $\|\cdot\|$ denotes the $\|.\|_2$ norm unless stated otherwise. Analogous to the complex case, every quaternion matrix $A \in \mathbb{Q}^{m \times n}$ admits a singular value decomposition, known as the quaternion singular value decomposition (QSVD):
\[
A = U \Sigma V^{H},
\]
where $U \in \mathbb{Q}^{m \times m}$ and $V \in \mathbb{Q}^{n \times n}$ are unitary matrices that satisfy $U^{H}U = I_m$ and $V^{H}V = I_n$, and $\Sigma \in \mathbb{R}_{+}^{m \times n}$ is a diagonal matrix whose nonnegative entries are the singular values of $A$.  
The QSVD extends the classical SVD to the quaternion setting.

\begin{definition}[\cite{MR4287902}]
Let $A \in \mathbb{Q}^{m \times n}$.  
A matrix $X \in \mathbb{Q}^{n \times m}$ is called the Moore--Penrose pseudoinverse of $A$, denoted by $A^{\dagger}$, if it satisfies the following four Penrose equations:
\[
\text{(a)}~ AXA = A, \quad 
\text{(b)}~ XAX = X, \quad 
\text{(c)}~ (AX)^{H} = AX, \quad 
\text{(d)}~ (XA)^{H} = XA.
\]
These conditions uniquely determine $A^{\dagger}$.
\end{definition}
In practice, the Moore--Penrose pseudoinverse is often computed using the QSVD~\cite{wei2018quaternion}.  
If $A = U \Sigma V^{H}$, where $U$ and $V$ are unitary and $\Sigma$ contains the singular values of $A$, then
\[
A^{\dagger} = V \Sigma^{\dagger} U^{H},
\]
where $\Sigma^{\dagger}$ is obtained by reciprocating the nonzero singular values of $\Sigma$ and transposing the result.  
This formulation provides a direct and numerically stable quaternionic extension of the classical SVD-based pseudoinverse that is used in real and complex domains.

\section{Quaternion hyperpower iterative method}\label{sec3}
In this section, we introduce a family of quaternion-based iterative schemes for computing the Moore--Penrose inverse of a quaternion matrix 
$A \in \mathbb{Q}^{m \times n}$. 
The core idea is to construct a sequence of matrices 
$\{X_j\}_{j=0}^{\infty}$ that converges to the desired pseudoinverse 
$A^{\dagger}$. 
We analyze the convergence behavior of these iterations and show that the choice of the initial approximation $X_0$ is crucial for ensuring numerical stability and achieving rapid convergence.

\begin{definition}
\label{def:qhp-matrix}
Let $A \in \mathbb{Q}^{m \times n}$ be a quaternion matrix.  
The quaternion hyperpower iterative method (QHPIM) is an iterative 
scheme for approximating the Moore--Penrose inverse $A^{\dagger} \in \mathbb{Q}^{n \times m}$. 
Starting from an initial approximation $X_0 \in \mathbb{Q}^{n \times m}$, 
the iteration is defined as
\begin{equation}\label{eq:qhp-matrix}
    X_{j+1} 
    = X_j \Big( I + R_j + R_j^2 + \cdots + R_j^{\,k-1} \Big), 
    \qquad 
    R_j = I - A X_j, 
    \qquad j = 0,1,2,\ldots,
\end{equation}
where $I \in \mathbb{Q}^{m \times m}$ is the identity matrix and 
$k \geq 1$ is an integer determining the number of residual terms 
included in each iteration. 
\end{definition}

The iterative process defined above generalizes the classical Newton--Schulz iteration to higher order. 
Its convergence behavior depends on the spectral properties of the residual matrix $R_j$, 
which measures how closely $A X_j$ approximates the projection onto the range of $A$. 
A careful analysis of this residual provides insight into both the rate of convergence 
and the numerical stability of the algorithm.

Let $R_j = I - A X_j$ denote the residual in the $j$th iteration.  
The quaternion hyperpower iterative method is said to have a convergence of order $k$ corresponding to the degree of the residual polynomial in \eqref{eq:qhp-matrix} - if there exists a constant $\alpha > 0$ such that
\begin{equation}\label{eq:qhp-order}
    \|R_{j+1}\| \;\leq\; \alpha \, \|R_j\|^k, 
    \quad j = 0,1,2,\ldots.
\end{equation}
The parameter $k$ directly influences the convergence rate of the iteration \eqref{eq:qhp-matrix}:  
\begin{itemize}
    \item For $k = 2$, the method reduces to the classical 
    Newton--Schulz (NS) iteration, which converges quadratically 
    under suitable spectral conditions.
    \item For $k > 2$, the scheme achieves higher-order convergence, driving the residual toward zero more rapidly at the expense of additional quaternion matrix multiplications per iteration. 
\end{itemize}
Hence, the quaternion hyperpower iteration provides a flexible framework that allows a trade-off between computational cost and convergence speed: smaller values of $k$ (e.g., $k = 2$) are computationally efficient but slower, while larger $k$ values yield faster convergence at higher per-iteration cost. To guaranty convergence of the iterative sequence $\{X_j\}$, 
the spectral radius of the initial residual must be less than one. 
The following theorem provides a sufficient condition on the scaling parameter $\alpha$ 
that ensures this property for initialization $X_0 = \alpha A^{H}$.
\begin{theorem}\label{initial}
Let \(A \in \mathbb{Q}^{m \times n}\) be a nonzero quaternion matrix of rank \(r \le \min(m,n)\), 
and \(\sigma_1 \ge \sigma_2 \ge \cdots \ge \sigma_r > 0\) denote its nonzero singular values. 
Let the initial approximation be \(X_0 = \alpha A^{H}\), with a scalar parameter \(\alpha > 0\). 
Define the initial residual relative to the Moore--Penrose condition as
\begin{equation}\label{eq:R0-def}
R_0 := P_{R(A)} - A X_0 = P_{R(A)} - \alpha A A^{H},
\end{equation}
where \(P_{R(A)}\) denotes the orthogonal projector onto the column space of \(A\). 
If \(0 < \alpha < \frac{2}{\sigma_1^2}\), then
\[
\rho(R_0) = \max_{1 \le i \le r} |1 - \alpha \sigma_i^2| < 1.
\]
\end{theorem}
\begin{proof}
By QSVD, there exist unitary matrices \(U \in \mathbb{Q}^{m \times m}\) and 
\(V \in \mathbb{Q}^{n \times n}\) such that
\begin{equation}\label{eq:qsvd}
A = U
\begin{bmatrix}
\Sigma & 0 \\
0 & 0
\end{bmatrix}
V^{H}, \text{where \(\Sigma = \mathrm{diag}(\sigma_1, \dots, \sigma_r)\) with 
\(\sigma_1 \ge \cdots \ge \sigma_r > 0\).}
\end{equation}  
From \eqref{eq:qsvd}, we obtain
\begin{equation}\label{eq:AAH}
A A^{H} = U
\begin{bmatrix}
\Sigma^2 & 0 \\
0 & 0
\end{bmatrix}
U^{H}.
\end{equation}
The orthogonal projector onto the column space of \(A\) is given by
\begin{equation}\label{eq:proj}
P_{R(A)} = U
\begin{bmatrix}
I_r & 0 \\
0 & 0
\end{bmatrix}
U^{H}.
\end{equation}
Substituting \eqref{eq:AAH} and \eqref{eq:proj} into \eqref{eq:R0-def}, we have
\[
R_0
= P_{R(A)} - \alpha A A^{H}
= U
\begin{bmatrix}
I_r - \alpha \Sigma^2 & 0 \\
0 & 0
\end{bmatrix}
U^{H}.
\]
Therefore, the eigenvalues of \(R_0\) are 
\(\{\,1 - \alpha \sigma_i^2\,\}_{i=1}^r\) together with \(0\) 
(of multiplicity \(m - r\)).  
Consequently,
\[
\rho(R_0) = \max_{1 \le i \le r} |1 - \alpha \sigma_i^2|.
\]
Finally, since \(0 < \alpha < 2 / \sigma_1^2\) implies 
\(|1 - \alpha \sigma_i^2| < 1\) for all \(i\), we conclude that 
\(\rho(R_0) < 1\), completing the proof.
\end{proof}
\begin{remark}
The residual \(R_j = I - A X_j\) defined in Definition~\ref{def:qhp-matrix} 
quantifies the deviation of the current iterate from the ideal Moore--Penrose condition. 
However, in Theorem~\ref{initial}, this residual is expressed as 
\(R_0 = P_{R(A)} - A X_0\) to account for the fact that, 
for rectangular or rank-deficient quaternion matrices, 
the product \(A X_j\) approximates the orthogonal projector \(P_{R(A)}\) onto the column space of \(A\), rather than the full identity matrix \(I_m\). 
This adjustment aligns with the Moore--Penrose property \(A A^{\dagger} = P_{R(A)}\) 
and enables a more accurate spectral characterization of the residual. 
When \(A\) is square and full rank, \(P_{R(A)} = I_m\), 
and both definitions coincide.
\end{remark}

\begin{remark}
The scaled adjoint initialization \(X_0 = \alpha A^{H}\) ensures a contractive 
starting point for the quaternion hyperpower iteration.  
From \eqref{eq:R0-def}, we have
$\rho\big(A(X_0 - A^{\dagger})\big) \;\le\; |1 - \alpha \sigma_1^2| \;<\; 1$,
which further implies
\[
\|A(X_0 - A^{\dagger})\| \;<\; 1.
\]
Thus, the initial approximation \(X_0\) lies within the region of convergence, 
guaranteeing that the iteration converges to \(A^{\dagger}\) whenever 
\(0 < \alpha < 2 / \sigma_1^2\).

The optimal choice of \(\alpha\) depends on the largest singular value 
\(\sigma_1\) of \(A\).  
Following the approach in~\cite{lebihan2004svd}, the singular values of a 
quaternion matrix can be computed via the SVD
of its complex representation \(A^C\).  
This provides a reliable means of determining \(\sigma_1\) and, consequently, 
an appropriate scaling parameter \(\alpha = 1 / \sigma_1^2\) for practical 
implementation of the quaternion hyperpower iteration.
\end{remark}

We next show that the quaternion hyperpower iteration preserves essential 
projector and Hermitian properties at every iteration step.  

\begin{proposition}\label{prop:quat-block}
Let the sequence $\{X_j\} \subseteq \mathbb{Q}^{n \times m}$ be generated by 
the quaternion hyperpower iterative method~\eqref{def:qhp-matrix} with the initial 
value $X_0 = \alpha A^{H}$, where $A \in \mathbb{Q}^{m \times n}$.  
Then, for all $j \ge 0$, the following relations hold:
\[
\text{(a)}\; A^{\dag} A X_j = X_j, \quad 
\text{(b)}\; X_j A A^{\dag} = X_j, \quad 
\text{(c)}\; (A X_j)^{H} = A X_j, \quad 
\text{(d)}\; (X_j A)^{H} = X_j A.
\]
\end{proposition}
\begin{proof}
We prove property (b); the remaining statements (a), (c), and (d) follow 
by analogous arguments using projector and Hermitian identities.

\textbf{Base case: (\(j = 0\)).}   
For \(X_0 = \alpha A^{H}\), we have $X_0 A A^{\dag} 
= \alpha A^{H} A A^{\dag} 
= \alpha A^{H} 
= X_0,$ where the Moore--Penrose identity \(A^{H} A A^{\dag} =A^{H}  (A^H)^{\dag} A^{H}\) is used.

\textbf{Inductive step:}  
Assume that for some \(N \ge 0\),
\begin{equation}\label{eq:ind-hyp}
X_N A A^{\dag} = X_N.
\end{equation}
From Definition~\ref{def:qhp-matrix}, the next iterate is given by
\begin{equation}\label{eq:iter}
X_{N+1} 
= X_N \Big( \sum_{s=0}^{k-1} R_N^s \Big),
\quad 
R_N = I - A X_N.
\end{equation}
We first examine the effect of the projector \(A A^{\dag}\) on \(R_N\):
\begin{equation}\label{eq:R-preserve}
R_N A A^{\dag} 
= (I - A X_N) A A^{\dag} 
= A A^{\dag} - A X_N A A^{\dag} 
= A A^{\dag} - A X_N 
= R_N,
\end{equation}
where the third equality follows from the induction hypothesis 
\eqref{eq:ind-hyp}.  
Consequently,
\[
R_N^s A A^{\dag} = R_N^s, \quad \text{for all } s \ge 0.
\]
Right-multiplying \eqref{eq:iter} by \(A A^{\dag}\) and applying the above identity gives
\[
X_{N+1} A A^{\dag}
= X_N \Big( \sum_{s=0}^{k-1} R_N^s \Big) A A^{\dag}
= \sum_{s=0}^{k-1} X_N R_N^s
= X_{N+1}.
\]
Hence, property $(b)$ holds for \(j = N+1\), and by induction, for all \(j \ge 0\).
\end{proof}
\begin{remark}
The key observation in the above proof is the invariance relation \eqref{eq:R-preserve}, which shows that the residual \(R_N\) remains unchanged under right multiplication by the projector \(A A^{\dag}\). This invariance propagates through the iteration, forcing all subsequent iterates to satisfy property $(b)$. Analogous reasoning with the left projector \(A^{\dag} A\) and Hermitian symmetry establishes Proposition $(a)$, $(c)$, and $(d)$.
\end{remark}
The above result confirms that the quaternion hyperpower iteration not only reduces the residual norm but also preserves the essential algebraic and geometric characteristics of the Moore--Penrose inverse. Each iterate \(X_j\) remains consistent with the orthogonal projectors \(P_{R(A)}\) and \(P_{R(A^{H})}\), ensuring that the sequence evolves entirely within the feasible subspace of valid pseudoinverses. This property establishes a solid foundation for developing higher-order quaternion iterative schemes, discussed in the following subsections.

\subsection{Quaternion RAPID (QRAPID)}
We now introduce a powerful and efficient iterative scheme, termed the QRAPID method, to calculate the Moore--Penrose inverse of a quaternion matrix \(A \in \mathbb{Q}^{m \times n}\). This approach extends the classical Newton iteration by combining it with a divided-difference strategy, resulting in a family of high-order methods whose
convergence order follows a Fibonacci-type sequence.  
The underlying idea is inspired by scalar root-finding techniques for the nonlinear function \(f(x) = \tfrac{1}{x} - a\), \(x \in \mathbb{R}\), as described 
in Algorithm~$1$ of~\cite{khosravi2023gibs}.  

The complete procedure for the proposed QRAPID iteration is summarized in Algorithm~\ref{alg:quat_rapid}. It generalizes the hyperpower framework introduced earlier and achieves accelerated convergence without sacrificing numerical stability.

\begin{algorithm}[ht]

\caption{QRAPID Algorithm}
\label{alg:quat_rapid}
\begin{algorithmic}[1]
\State \textbf{Input:} Quaternion matrix 
$A = A^{(s)} + A^{(x)}\mathbf{i} + A^{(y)}\mathbf{j} + A^{(z)}\mathbf{k} \in \mathbb{Q}^{m\times n}$;  
step parameter $N \ge 0$; tolerance $\mathtt{tol}$;

\State \textbf{Scaling:} Choose $\alpha > 0$ via the complex representation
\[
A^C = 
\begin{bmatrix}
A^{(s)}+A^{(x)}\mathbf{i} & A^{(y)}+A^{(z)}\mathbf{i} \\
-\,(A^{(y)}-A^{(z)}\mathbf{i}) & A^{(s)}-A^{(x)}\mathbf{i}
\end{bmatrix},
\quad 
\alpha = \frac{1}{\big(\sigma_{\max}(A^C)\big)^2};
\]

\State \textbf{Initialization:} 
$X_0 = \alpha A^{H}$;  
$I = I_m$;  
$j \leftarrow 0$;

\While{true}

  \State \textbf{Residual product:} 
  $P_j \leftarrow A X_j$;

  \State \textbf{Auxiliary update:} 
  $U_j \leftarrow \tfrac{1}{4}\,X_j\Big(13I - P_j\big(15I - P_j(7I - P_j)\big)\Big)$;

  \State \textbf{Correction:} 
  $V_j \leftarrow U_j + X_j\,(I - A U_j)$;

  \If{$N=0$}
    \State $Z_N \leftarrow V_j$;
  \Else
    \State $Y_0 \leftarrow U_j$; $W_0 \leftarrow V_j$;
    \For{$l = 1$ to $N$}
      \State $Z_l \leftarrow W_{l-1} + Y_{l-1}\,(I - A W_{l-1})$;
      \State $Y_l \leftarrow W_{l-1}$; \quad $W_l \leftarrow Z_l$;
    \EndFor
    \State $Z_N \leftarrow Z_l$ (final iterate);
  \EndIf

  \State \textbf{Main update:} 
  $X_{j+1} \leftarrow Z_N + X_j\,(I - A Z_N)$;

  \State \textbf{Stopping test:} 
  If $\|X_{j+1} - X_j\|_F < \mathtt{tol}$, \textbf{break};

  \State $j \leftarrow j+1$;

\EndWhile

\State \textbf{Output:} $X_{\mathrm{final}} = X_{j+1}$ (approximation of $A^\dagger$);

\end{algorithmic}
\end{algorithm}

The order of convergence of the QRAPID iteration depends on the parameter \(N\), which controls the number of nested updates within each outer iteration. For compactness, we denote by \(X_{j,N}\) the quaternion iterate obtained at the \(j\)th step with parameter \(N\). 
\begin{theorem}\label{thm:convrapidQ}
Let $A \in \mathbb{Q}^{m\times n}$ be a quaternion matrix with nonzero singular 
values $\sigma_1 \geq \sigma_2 \geq \cdots \geq \sigma_r > 0$, and choose the initial guess
\[
X_{0,N} = \alpha A^{H}, \quad 0 < \alpha < \tfrac{2}{\sigma_1^2},
\]
where $\sigma_1$ is the largest singular value of $A$.  
Then the quaternion sequence $\{X_{j,N}\}$ for $N+5$ step generated by 
Algorithm~\ref{alg:quat_rapid} converges to the Moore--Penrose inverse 
\(A^{\dagger}\), with the convergence order characterized by
\[
\|X_{j+1,N} - A^{\dagger}\| \;\leq\;
\begin{cases}
  \|A^{\dagger}\|\,\|A\|^{5}\,\|X_{j,N} - A^{\dagger}\|^{5}, & N=0, \\[1ex]
  \|A^{\dagger}\|\,\|A\|^{4(N+1)}\,\|X_{j,N} - A^{\dagger}\|^{4(N+1)}, & N=1,2, \\[1ex]
  \|A^{\dagger}\|\,\|A\|^{\nu_N}\,\|X_{j,N} - A^{\dagger}\|^{\nu_N}, & N \geq 3,
\end{cases}
\]
where
\[
\nu_N = 11a_{N-1}+7a_{N-2}+1, \quad
a_N = \frac{\phi^N-(1-\phi)^N}{\sqrt{5}}, \quad
\phi=\frac{1+\sqrt{5}}{2}.
\]
\end{theorem}
\begin{proof}
We outline the main steps of the proof for different values of the step parameter \(N\).  
Let \(E_{j,N} = X_{j,N} - A^{\dagger}\) and \(R_{j,N} = I - A X_{j,N}\).

\noindent\textbf{Case $N=0$.}  
Using projector properties of the Moore--Penrose inverse, we have
\[
(I - A A^{\dagger})^{j} = I - A A^{\dagger}, 
\quad (I - A A^{\dagger})\,A\,E_{j,0} = 0, 
\quad j=1,2,\ldots .
\]
The RAPID update can be expressed as
\begin{equation}\label{eq:update-N0}
U_{j,0} = X_{j,0}(I + R_{j,0} + R_{j,0}^2) + \tfrac{1}{4} X_{j,0} R_{j,0}^3,
\end{equation}
which after substitution leads to
\begin{equation}\label{eq:R-N0}
R_{j+1,0} = \tfrac{3}{4} R_{j,0}^5 + \tfrac{1}{4} R_{j,0}^6 .
\end{equation}
Multiplying \eqref{eq:R-N0} by $A$ gives
\begin{equation}\label{eq:E-N0}
A E_{j+1,0} = -\tfrac{3}{4}(A E_{j,0})^{5} - \tfrac{1}{4}(A E_{j,0})^{6}.
\end{equation}
From \eqref{eq:E-N0}, taking norms and assuming $\|A E_{0,0}\| < 1$, it follows that
\[
\|E_{j+1,0}\| \leq \|A^{\dagger}\|\,\|A\|^{5} \|E_{j,0}\|^{5},
\]
demonstrating fifth-order convergence.

\noindent\textbf{Case $N=1$.}  
A similar expansion gives
\[
R_{j+1,1} = \tfrac{1}{16}\big(9R_{j,1}^8 + 6R_{j,1}^9 + R_{j,1}^{10}\big).
\]
Consequently,
\begin{equation}\label{eq:E-N1}
A E_{j+1,1} = -\tfrac{9}{16}(A E_{j,1})^{8} 
              + \tfrac{6}{16}(A E_{j,1})^{9} 
              - \tfrac{1}{16}(A E_{j,1})^{10}.
\end{equation}
Using \eqref{eq:E-N1} and $\|A E_{0,1}\| < 1$, it follows that
$
\|E_{j+1,1}\| \leq \|A^{\dagger}\|\,\|A\|\,\|E_{j,1}\|^{8},
$ so the convergence order is $8$.  The case $N=2$ proceeds analogously, leading to
$ \|E_{j+1,2}\| \leq \|A^{\dagger}\|\,\|A\|^{12}\|E_{j,2}\|^{12},
$ showing the order $12$.

\noindent\textbf{Case $N \geq 3$.}  
By induction, one shows that
\begin{equation}\label{eq:R-general}
R_{j+1,N} = \frac{1}{4^{a_{N+2}}} (3I + R_{j,N})^{a_{N+2}} R_{j,N}^{\nu_N},
\end{equation}
where $\nu_N = 11a_{N-1} + 7a_{N-2} + 1$.  
From \eqref{eq:R-general} it follows directly that
\[
\|E_{j+1,N}\| \leq \|A^{\dagger}\|\,\|A\|^{\nu_N} \|E_{j,N}\|^{\nu_N}.
\]
Thus, in all cases, the sequence $\{X_{j,N}\}$ converges to $A^{\dagger}$ with the claimed convergence order.
\end{proof}
We next analyze the stability of the QRAPID iteration in the presence of small perturbations. 
In particular, we establish how a small error introduced in iteration $j$ propagates
to the next iterate.
\begin{theorem}\label{thm:pertQrapid}
Let \(\{X_j\}\) be the sequence generated by the QRAPID 
Algorithm~\ref{alg:quat_rapid}, under the assumptions of 
Theorem~\ref{thm:convrapidQ}.  
Suppose that a perturbation occurs at iteration \(j\), such that 
\(\Delta X_j = \tilde{X}_j - X_j\), where \(\Delta X_j\) is sufficiently small 
that terms of order \(\mathrm{O}(\|\Delta X_j\|^2)\) and higher may be neglected.  
Then, for \(N=5\) (the ten–step RAPID method with convergence order 
\(\nu_5=48\)), the perturbation satisfies
\[
\|\Delta X_{j+1}\| 
\;\leq\; 
\nu_5\,\|\Delta X_j\|\,\|R_j\|^{\nu_5-1}
\Big[\,2\max\{(3/2)^{a_7},\,\|R_j/2\|^{a_7}\}\,\Big]
+ \mathrm{O}(\|\Delta X_j\|),
\]
where \(R_j = I - A X_j\), 
\(a_7=\tfrac{\phi^7-(1-\phi)^7}{\sqrt{5}}\), 
and \(\phi=\tfrac{1+\sqrt{5}}{2}\).
\end{theorem}
\begin{proof}
For \(N=5\), the residual recurrence~\eqref{eq:R-general} becomes
\begin{equation}\label{eq:R5}
  R_{j+1} = \frac{1}{4^{a_7}}(3I+R_j)^{a_7} R_j^{\nu_5}, 
  \quad \nu_5=48.
\end{equation}
Let the perturbed residual be \(\tilde{R}_j = I - A\tilde{X}_j = R_j - A\Delta X_j\). Expanding as in~\eqref{eq:R5}, one obtains
\begin{align}
  I - A\tilde{X}_{j+1} 
   &= \frac{1}{4^{a_7}} \sum_{k=0}^{13} 3^{13-k}\binom{13}{k}\,\tilde{R}_j^{k+48}, 
   \label{eq:tildeR-exp}\\
  I - AX_{j+1} 
   &= \frac{1}{4^{a_7}} \sum_{k=0}^{13} 3^{13-k}\binom{13}{k}\,R_j^{k+48}.
   \label{eq:R-exp}
\end{align}
Subtracting \eqref{eq:R-exp} from \eqref{eq:tildeR-exp} gives
\[
  A\tilde{X}_{j+1} - AX_{j+1} 
   = -\frac{1}{4^{a_7}}\sum_{k=0}^{13}3^{13-k}\binom{13}{k}
     \left(\tilde{R}_j^{k+48}-R_j^{k+48}\right).
\]
Multiplying by $A^\dagger$ and using $A^\dagger A\Delta X_{j+1}=\Delta X_{j+1}$, 
we arrive at
\begin{align}\label{eq:deltaX}
  \Delta X_{j+1} &= \frac{1}{4^{a_7}}\sum_{k=0}^{13}3^{13-k}\binom{13}{k}\,
  \Delta X_j\Big(\tilde{R}_j^{k+47}+\tilde{R}_j^{k+46}R_j+\cdots+\tilde{R}_jR_j^{k+46}+R_j^{k+47}\Big).
\end{align}
Now, observe that
\[
  \|\tilde{R}_j\|^k = \|R_j - A\Delta X_j\|^k \;\leq\;(\|R_j\|+\|A\Delta X_j\|)^k.
\]
Taking norms in~\eqref{eq:deltaX}, and using that $\Delta X_j$ is sufficiently small so that higher-order terms can be neglected, yields
\[
  \|\Delta X_{j+1}\|
   \leq \|\Delta X_j\|\,\|R_j\|^{47}
        \left[\frac{a_7+\nu_5}{4^{a_7}}(3+\|R_j\|)^{13}\right]
        +\mathrm{O}(\|\Delta X_j\|).
\]
Finally, applying the inequality $(x+y)^{n}\leq 2^{\,n-1}(x^n+y^n)$ for $x,y>0$, 
and recalling $a_7<\nu_5$, we simplify to
\[
  \|\Delta X_{j+1}\|
   \leq \nu_5\,\|\Delta X_j\|\,\|R_j\|^{\nu_5-1}
        \Big[2\max\{(3/2)^{a_7},\|R_j/2\|^{a_7}\}\Big]
        +\mathrm{O}(\|\Delta X_j\|).
\]
This completes the proof.
\end{proof}
The above result confirms that the QRAPID iteration retains strong stability even under small perturbations, with the amplification factor bounded by the high-order residual term \(\|R_j\|^{\nu_5-1}\). This makes the method well-suited for large-scale or ill-conditioned quaternion systems, where numerical robustness is essential.

\subsection{Quaternion strong approximate inverse (QSAI)}
This subsection presents the QSAI iterative scheme. The method exploits quaternion multiplication to compute generalized inverses of quaternion matrices with high numerical stability. The principal idea is to reformulate the classical hyperpower-type iteration to achieve tenth-order convergence while reducing the number of quaternion matrix multiplications per iteration.

Let $X_{0}$ denote the initial approximation and $R_{j}$ the corresponding residual matrix in iteration $j$. The next iterate $X_{j+1}$ is computed as
 \[
X_{j+1} = X_{j}\big(I + R_{j} + \cdots + R_{j}^9\big)
= X_{j}(I + R_{j})(I + \beta_1 R_{j}^2 + R_{j}^4)(I + \beta_2 R_{j}^2 + R_{j}^4),
\]
where the parameters $\beta_1$ and $\beta_2$ satisfy $\beta_1+\beta_2=1$ and $\beta_1\beta_2=-1$, leading to the explicit solutions $\beta_1=\tfrac{1+\sqrt{5}}{2}$ and $\beta_2=\tfrac{1-\sqrt{5}}{2}$. This factorized formulation reduces the computational cost by limiting each iteration to six quaternion matrix multiplications, while preserving the desirable properties of high-order convergence and numerical stability. The complete procedure of the QSAI algorithms is summarized in Algorithm~\ref{alg:q_sai}.
\begin{algorithm}[ht]
\caption{QSAI Algorithm}
\label{alg:q_sai}
\begin{algorithmic}[1]

\State \textbf{Input:} Quaternion matrix 
$A = A^{(s)} + A^{(x)}\mathbf{i} + A^{(y)}\mathbf{j} + A^{(z)}\mathbf{k} \in \mathbb{Q}^{m\times n}$;  
golden ratio scalars $\beta_1=\tfrac{1+\sqrt{5}}{2}$, $\beta_2=\tfrac{1-\sqrt{5}}{2}$;  
tolerance $\mathtt{tol}$;

\State \textbf{Scaling:} Choose $\alpha>0$ via the complex representation
\[
A^C = 
\begin{bmatrix}
A^{(s)}+A^{(x)}\mathbf{i} & A^{(y)}+A^{(z)}\mathbf{i} \\
-\,(A^{(y)}-A^{(z)}\mathbf{i}) & A^{(s)}-A^{(x)}\mathbf{i}
\end{bmatrix},
\quad 
\alpha = \frac{1}{\big(\sigma_{\max}(A^C)\big)^2};
\]

\State \textbf{Initialization:} 
$X_0 = \alpha A^{H} \in \mathbb{Q}^{n\times m}$; 
$I = I_m$; 
$j \leftarrow 0$;

\While{true}

   \State \textbf{Residual:} $R_j \leftarrow I - A X_j \in \mathbb{Q}^{m\times m}$;

   \State \textbf{Polynomial construction:} 
   $Q_j \leftarrow (I + \beta_1 R_j^{2} + R_j^{4})(I + \beta_2 R_j^{2} + R_j^{4})$;

   \State \textbf{Update:} 
   $X_{j+1} \leftarrow X_j\,(I+R_j)\,Q_j$;

   \State \textbf{Stopping test:} 
   If $\|X_{j+1}-X_j\|_F < \mathtt{tol}$, \textbf{break};

   \State $j \leftarrow j+1$;

\EndWhile

\State \textbf{Output:} $X_{\mathrm{final}} = X_{j+1}$ (approximation of $A^\dagger$);

\end{algorithmic}
\end{algorithm}

We now establish the convergence behavior of the proposed QSAI algorithm.
\begin{theorem}\label{prop:strong1-mat}
Let $A \in \mathbb{Q}^{m\times n}$ be a quaternion matrix with nonzero singular values 
$\sigma_1 \ge \sigma_2 \ge \dots \ge \sigma_r > 0$. 
Consider the initial approximation
\[
X_0 = \alpha A^H, 
\quad 0 < \alpha < \frac{2}{\sigma_1^2},
\]
where $\sigma_1$ denotes the largest singular value of $A$. 
Let $\{X_j\}$ be the sequence generated by the QSAI iteration. 
Then $\{X_j\}$ converges to the Moore--Penrose inverse $A^\dagger$ of $A$, 
and the convergence order satisfies
\[
\|X_{j+1} - A^\dagger\| \le 
\|A^\dagger\|\,\|A\|^{10}\,\|X_j - A^\dagger\|^{10}.
\]
\end{theorem}
\begin{proof}
Define the error matrix $E_j = X_j - A^\dagger$. 
From the Moore--Penrose conditions, it follows that
\begin{equation}\label{eqn:pm4-mat}
(I - A A^\dagger)^j = I - A A^\dagger, 
\quad (I - A A^\dagger) A E_j = 0.
\end{equation}
Substituting into the QSAI update formula yields
\begin{align}\label{eqn:pm41-mat}
A E_{j+1} = -\!\left(I - A A^\dagger - A E_j\right)^{10} + I - A A^\dagger 
= - (A E_j)^{10}.
\end{align}
Assuming $\|A E_0\| < 1$, taking norms in \eqref{eqn:pm41-mat} gives
\[
\|A E_{j+1}\| \le \|A E_j\|^{10} 
\le \|A\|^{10}\, \|E_j\|^{10}.
\]
Furthermore, since $E_{j+1} = A^\dagger (A E_{j+1})$, we obtain
\[
\|E_{j+1}\| 
\le \|A^\dagger\|\,\|A E_{j+1}\| 
\le \|A^\dagger\|\,\|A\|^{10}\,\|E_j\|^{10}.
\]
Hence, the sequence $\{X_j\}$ converges to $A^\dagger$ with tenth-order convergence.
\end{proof}
This result confirms that the QSAI method retains the same order of convergence as the full tenth-degree hyperpower iteration while significantly reducing computational cost through factorization.  

We now analyze the numerical stability of the QSAI iteration.  
In particular, we study how a small perturbation introduced at one iteration step affects the subsequent iterates.  
This analysis provides insight into the robustness of the QSAI algorithm when implemented in finite-precision arithmetic.
\begin{theorem}\label{thm:pm-perturb-matrix}
Let \(A \in \mathbb{Q}^{m\times n}\), and let \(\{X_j\}\) be the sequence generated by the QSAI Algorithm~\ref{alg:q_sai} under the assumptions of Theorem~\ref{prop:strong1-mat}.  
Suppose that at iteration \(j\) a small perturbation \(\Delta X_j = \tilde{X}_j - X_j\) occurs, where \(\Delta X_j\) is sufficiently small such that terms of order \(\mathrm{O}(\|\Delta X_j\|^2)\) and higher can be neglected.  
Then the perturbation at the next iteration satisfies
\begin{equation*}
\|\Delta X_{j+1}\|
\;\le\; \Gamma\,\|\Delta X_0\|,
\quad
\Gamma \;=\; 10^{\,j+1}\prod_{i=0}^{j}
\max\!\big\{1,\,\|R_i\|^{9}\big\}\,
\Big(1+9\,\|A\|\,\|X_i\|\Big), \text{~where \(R_i = I - A X_i\).}
\end{equation*}

\end{theorem}
\begin{proof}
Let the perturbed iterate be \(\tilde{X}_j = X_j + \Delta X_j\), and define the corresponding perturbed residual as
\begin{equation}\label{eq:pert-residual}
\tilde{R}_j := I - A\tilde{X}_j = R_j - A\Delta X_j,
\quad R_j := I - A X_j.
\end{equation}
From the QSAI update formula, we have
\begin{equation*}
\tilde{X}_{j+1} = \tilde{X}_j \Big(\sum_{k=0}^{9}\tilde{R}_j^{\,k}\Big),
\quad
X_{j+1} = X_j \Big(\sum_{k=0}^{9}R_j^{\,k}\Big),
\end{equation*}
which yields the exact difference recurrence
\begin{equation}\label{eq:deltaX-rec}
\Delta X_{j+1}
= (\Delta X_j)\Big(\sum_{k=0}^{9}\tilde{R}_j^{\,k}\Big)
  + X_j\Big(\sum_{k=0}^{9}(\tilde{R}_j^{\,k}-R_j^{\,k})\Big).
\end{equation}
To bound each term, we first use \eqref{eq:pert-residual} and submultiplicativity to write
\begin{equation}\label{eq:Phi-def}
\|\tilde{R}_j^k\| \le (\|R_j\|+\|A\|\,\|\Delta X_j\|)^k = \Phi_j^{\,k}.
\end{equation}
Applying the telescoping identity and retaining only first-order terms in \(\Delta X_j\) gives
\begin{equation}\label{eq:bound-diff}
\|\tilde{R}_j^{\,k}-R_j^{\,k}\| \le k\,\|R_j\|^{\,k-1}\,\|A\|\,\|\Delta X_j\|.
\end{equation}
Substituting \eqref{eq:Phi-def} and \eqref{eq:bound-diff} into \eqref{eq:deltaX-rec} and taking norms, we obtain
\[
\|\Delta X_{j+1}\|
\le \|\Delta X_j\|\sum_{k=0}^{9}\Phi_j^{\,k}
+ \|X_j\|\,\|A\|\,\|\Delta X_j\|\sum_{k=1}^{9}k\,\|R_j\|^{\,k-1}
+ \mathrm{O}(\|\Delta X_j\|).
\]
Using the bounds $\sum_{k=0}^{9}\Phi_j^{\,k}\le 10\max\{1,\|R_j\|^9\}$ and
$\sum_{k=1}^{9}k\,\|R_j\|^{\,k-1}\le 9\max\{1,\|R_j\|^9\}$ yields the single-step
estimate
\begin{equation}\label{eq:final-step}
\|\Delta X_{j+1}\|
\le 10\,\|\Delta X_j\|\,\max\{1,\|R_j\|^9\}\,\big(1+9\,\|A\|\,\|X_j\|\big)
+ \mathrm{O}(\|\Delta X_j\|).
\end{equation}
Iterating inequality \eqref{eq:final-step} from \(0\) to \(j\) gives
\[
\|\Delta X_{j+1}\|
\le 10^{\,j+1}\Big(\prod_{i=0}^{j}\max\{1,\|R_i\|^9\}\,(1+9\,\|A\|\,\|X_i\|)\Big)\|\Delta X_0\|
+ \text{higher-order terms},
\]
which establishes the stated bound.
\end{proof}

\subsection{Quaternion hyperpower iteration of order $19$ (QHPI$19$)}
In this subsection, we present QHPI$19$ method for computing the Moore--Penrose inverse of a quaternion matrix $A\in\mathbb{Q}^{m\times n}$. Let the residual at iteration $j$ be $R_j := I - A X_j.$
The direct $p=19$ hyperpower update is given by
\[
X_{j+1} = X_j \sum_{k=0}^{18} R_j^{\,k},
\]
which guarantees $19$th–order convergence but requires up to nineteen quaternion matrix–matrix products per iteration. 
This high computational cost can become prohibitive for large–scale matrices. To address this issue, we develop an equivalent factorized form that preserves the convergence order while substantially reducing the number of required multiplications per iteration.

\noindent\textbf{Factorized update.}
To obtain a more efficient implementation, the direct expansion is reorganized into a compact structure using lower–degree polynomial blocks. 
The update can be expressed as
\begin{equation}\label{eq:QHPI19-main}
X_{j+1} \;=\; X_j \Big[\, I + (R_j + R_j^2)\,\Gamma_j \,\Big],
\end{equation}
where 
\[
\Gamma_j \;=\; I + R_j^2 + R_j^4 + R_j^6 + R_j^8 
              + R_j^{10} + R_j^{12} + R_j^{14} + R_j^{16}.
\]
This representation avoids explicitly computing all powers up to $R_j^{18}$ 
and instead expresses the iteration through nested polynomial structures that 
reuse lower powers of $R_j$. 
This organization significantly reduces computational redundancy and enhances numerical efficiency.

\noindent\textbf{Polynomial factorization.}
The polynomial $\Gamma_j$ in \eqref{eq:QHPI19-main} can be equivalently represented 
as a product of two even polynomials with an additional correction term:
\begin{align}\label{eq:QHPI19-polyfact}
\Gamma_j 
&= \big(I + a_1 R_j^2 + a_2 R_j^4 + a_3 R_j^6 + R_j^8\big)\,
   \big(I + b_1 R_j^2 + b_2 R_j^4 + b_3 R_j^6 + R_j^8\big)
   \;+\; \big(c_1 R_j^2 + c_2 R_j^4\big),
\end{align}
with the constraint $a_3=b_3$.  
The coefficients $a_i, b_i,$ and $c_i$ are determined such that the expansion of
\eqref{eq:QHPI19-polyfact} reproduces all even powers of $R_j$ up to $R_j^{16}$.
Matching coefficients gives rise to the following system of linear equations:
\begin{equation}\label{eq:coeff-system}
\left.
\begin{aligned}
a_1 + b_1 + c_1 &= 1, \\
a_2 + c_2 + a_1 b_1 + b_2 &= 1, \\
2a_3 + a_2 b_1 + a_1 b_2 &= 1, \\
2 + a_1 a_3 + a_3 b_1 + a_2 b_2 &= 1, \\
a_1 + a_2 a_3 + b_1 + a_3 b_2 &= 1, \\
a_2 + a_3^2 + b_2 &= 1, \\
2a_3 &= 1,
\end{aligned}
\right\}
\end{equation}
whose solution is
\begin{equation}
\left.
\begin{aligned}
a_1 &= \dfrac{5}{496}(31+\sqrt{93}), 
&\quad a_2 &= \dfrac{1}{8}(3+\sqrt{93}), 
&\quad a_3 &= \dfrac{1}{2}, \\[1ex]
b_1 &= -\dfrac{5}{496}(\sqrt{93}-31), 
&\quad b_2 &= \dfrac{1}{8}(3-\sqrt{93}), \\[1ex]
c_1 &= \dfrac{3}{8}, 
&\quad c_2 &= \dfrac{321}{1984}.
\end{aligned}
\right\}.
\end{equation}

\noindent\textbf{Secondary factorization.}
To further reduce the computational effort, each eighth–degree polynomial in \eqref{eq:QHPI19-polyfact} can be factored into quadratic blocks. Specifically,
\begin{equation}\label{eq:QHPI19-g1}
\Gamma_{1,j} := I + a_1 R_j^2 + a_2 R_j^4 + a_3 R_j^6 + R_j^8
= \underbrace{\big(I + d_1 R_j^2 + R_j^4\big)\,\big(I + d_2 R_j^2 + R_j^4\big)}_{\text{two quadratics}}
  \;+\; d_3 R_j^2,
\end{equation}
and
\begin{equation}\label{eq:QHPI19-g2}
\Gamma_{2,j} := I + b_1 R_j^2 + b_2 R_j^4 + b_3 R_j^6 + R_j^8
= \big(I + d_1 R_j^2 + R_j^4\big)\,\big(I + d_2 R_j^2 + R_j^4\big)
  \;+\; \big(e_1 R_j^2 + e_2 R_j^4\big),
\end{equation}
where the parameters $(d_1,d_2,d_3)$ are computed once and reused in both factorizations:
\[
d_1=\tfrac{1}{4}\big(\sqrt{27-2\sqrt{93}}+1\big),\quad
d_2=\tfrac{1}{4}\big(1-\sqrt{27-2\sqrt{93}}\big),\quad
d_3=\tfrac{1}{496}\big(5\sqrt{93}-93\big),
\]
and
\[
e_1=\tfrac{1}{496}\big(-93-5\sqrt{93}\big),\qquad
e_2=-\tfrac{\sqrt{93}}{4}.
\]
Hence, the compact representation of $\Gamma_j$ is
\[
\Gamma_j = \Gamma_{1,j}\,\Gamma_{2,j} + c_1 R_j^2 + c_2 R_j^4.
\]

\medskip
\noindent\textbf{Efficiency remark.}
The direct $19$th–order hyperpower expansion requires nineteen 
quaternion matrix multiplications per iteration. In contrast, the above factorized formulation reduces the cost to only $7$ multiplications per iteration while maintaining the same $19$th–order convergence. This significant reduction makes the QHPI$19$ method highly efficient and well–suited for large–scale quaternion matrix computations.

The complete QHPI$19$ algorithm is summarized in 
Algorithm~\ref{alg:q_hpi19_direct}.
\begin{algorithm}[ht]

\caption{QHPI$19$ Algorithm}
\label{alg:q_hpi19_direct}
\begin{algorithmic}[1]

\State \textbf{Input:} Quaternion matrix 
$A = A^{(s)} + A^{(x)}\mathbf{i} + A^{(y)}\mathbf{j} + A^{(z)}\mathbf{k} \in \mathbb{Q}^{m\times n}$;  
tolerance $\mathtt{tol}$; polynomial coefficients 
$a_i,b_i,c_i,d_i,e_i$ as defined in 
\eqref{eq:coeff-system}--\eqref{eq:QHPI19-g2};

\State \textbf{Scaling:} Choose $\alpha>0$ via the complex representation
\[
A^C = 
\begin{bmatrix}
A^{(s)}+A^{(x)}\mathbf{i} & A^{(y)}+A^{(z)}\mathbf{i} \\
-\,(A^{(y)}-A^{(z)}\mathbf{i}) & A^{(s)}-A^{(x)}\mathbf{i}
\end{bmatrix} \in \mathbb{C}^{2m\times 2n},
\quad 
\alpha = \frac{1}{\big(\sigma_{\max}(A^C)\big)^2};
\]

\State \textbf{Initialization:} 
$X_0 = \alpha A^{H} \in \mathbb{Q}^{n\times m}$; 
$I = I_m$ ; 
$j \leftarrow 0$;

\While{true}

   \State \textbf{Residual powers:} 
   $R_j \leftarrow I - A X_j$; \quad 
   $R_j^2 \leftarrow R_j R_j$; \quad 
   $R_j^4 \leftarrow R_j^2 R_j^2$; \quad 
   $R_j^8 \leftarrow R_j^4 R_j^4$; \quad 
   $R_j^{16} \leftarrow R_j^8 R_j^8$;

   \State \textbf{Quadratic blocks:}
   $U_j \leftarrow (I+d_1 R_j^2+R_j^4)(I+d_2 R_j^2+R_j^4)$; \quad 
   $V_j \leftarrow U_j+d_3 R_j^2$; \quad
   $W_j \leftarrow U_j+e_1 R_j^2+e_2 R_j^4$;

   \State \textbf{Polynomial factorization:}
   $\Gamma_j \leftarrow V_j W_j + c_1 R_j^2 + c_2 R_j^4$;

   \State \textbf{Update:}
   $X_{j+1} \leftarrow X_j \big(I + (R_j+R_j^2)\Gamma_j\big)$;

   \State \textbf{Stopping test:} 
   If $\|X_{j+1}-X_j\|_F < \mathtt{tol}$, \textbf{break};

   \State $j \leftarrow j+1$;

\EndWhile

\State \textbf{Output:} $X_{\mathrm{final}} = X_{j+1}$ 
(an approximation of $A^\dagger$);
\end{algorithmic}
\end{algorithm}

The convergence behavior of the $19$th--order quaternion hyperpower iteration is summarized in the following result.  
\begin{theorem}\label{thm:quat_hyper19_conv}
Let $A \in \mathbb{Q}^{m\times n}$ be a quaternion matrix with nonzero singular 
values $\sigma_1 \geq \sigma_2 \geq \cdots \geq \sigma_r > 0$, and choose the initial iterate
\[
X_0 = \alpha A^H, 
\quad 0 < \alpha < \tfrac{2}{\sigma_1^2},
\]
where $\sigma_1$ is the largest singular value of $A$.  
Let $\{X_j\}$ be the sequence generated by the QHPI$19$ iteration (Algorithm~\ref{alg:q_hpi19_direct}). Then $X_j$ converges to the Moore--Penrose inverse $A^\dagger$ of $A$, and the convergence order is given by
\[
\|X_{j+1} - A^\dagger\| \;\le\; \|A^\dagger\|\,\|A\|^{19}\,\|X_j - A^\dagger\|^{19}.
\]
\end{theorem}
\begin{proof}
Let the approximation error at iteration $j$ be denoted by $E_j := X_j - A^\dagger,$ and define the corresponding residual as $R_j := I - A X_j$. Using the update rule of the QHPI$19$ method, the next iterate can be expressed as
\begin{equation}\label{eq:qhp19-update}
    X_{j+1} \;=\; X_j \Big(I + (R_j+R_j^2)\big(V_j W_j + c_1 R_j^2 + c_2 R_j^4\big)\Big),
\end{equation}
where $U_j$, $V_j$, and $W_j$ denote the polynomial blocks constructed in Algorithm~\ref{alg:q_hpi19_direct}.  

By direct algebraic expansion, the residual satisfies
\begin{equation}\label{eq:qhp19-residual}
    R_{j+1} \;=\; R_j^{19}.
\end{equation}
Using \eqref{eq:qhp19-residual}, we obtain
\[
A E_{j+1} \;=\; A(X_{j+1} - A^\dagger)
= -R_{j+1} + (I - A A^\dagger)
= -R_j^{19} + (I - A A^\dagger).
\]
Since $I - A A^\dagger = 0$ on the range of $A$, the above expression simplifies to
\[
A E_{j+1} = -(A E_j)^{19}.
\]
Taking the norm on both sides and applying the submultiplicative property of matrix norms, we obtain
\[
\|A E_{j+1}\| \;\le\; \|A\|^{19}\,\|E_j\|^{19}.
\]
Finally, premultiplying by $A^\dagger$ and using 
$\|A^\dagger B\| \le \|A^\dagger\|\,\|B\|$ gives
\[
\|E_{j+1}\|
= \|A^\dagger (A E_{j+1})\|
\;\le\; \|A^\dagger\|\,\|A\|^{19}\,\|E_j\|^{19}.
\]
Hence, the sequence $\{X_j\}$ converges to $A^\dagger$ with local order $19$.
\end{proof}
The stability properties of the proposed QHPI$19$ method under small numerical perturbations are analyzed in the following theorem. This result demonstrates how local errors introduced during the iteration propagate through subsequent steps, providing an upper bound on the amplification of such perturbations. 
\begin{theorem}\label{thm:quat_hyper19_perturb}
Let $A \in \mathbb{Q}^{m\times n}$, and let $\{X_j\}$ denote the sequence generated by the QHPI$19$ Algorithm~\ref{alg:q_hpi19_direct} under the assumptions of Theorem~\ref{thm:quat_hyper19_conv}.  
Suppose that at the $j$th iteration, a small numerical perturbation occurs such that $\tilde{X}_j = X_j + \Delta X_j$, where $\Delta X_j$ is sufficiently small so that terms of quadratic or higher order in $\Delta X_j$ may be neglected.  
Then the perturbation at the next iteration satisfies
\begin{equation*}
\|\Delta X_{j+1}\|
\;\le\; \Gamma\,\|\Delta X_0\|,
\quad
\Gamma \;=\; 19^{\,j+1}\prod_{i=0}^{j}\max\!\big\{1,\;\|R_i\|^{18}\big\}\,\Big(1+18\,\|A\|\,\|X_i\|\Big),~\text{where $R_i = I - A X_i$.}
\end{equation*}
\end{theorem}
\begin{proof}
Let $\tilde{X}_j = X_j + \Delta X_j$ be the perturbed iterate.  
The corresponding perturbed residual is then given by
\begin{equation}\label{eq:pert-residual-19}
\tilde{R}_j \;:=\; I - A\tilde{X}_j \;=\; R_j - A\Delta X_j.
\end{equation}
The QHPI$19$ update rule for both the exact and perturbed iterates can be expressed as
\begin{equation}\label{eq:updates-19}
\tilde{X}_{j+1} = \tilde{X}_j \Big(\sum_{k=0}^{18}\tilde{R}_j^{\,k}\Big),
\quad
X_{j+1} = X_j \Big(\sum_{k=0}^{18}R_j^{\,k}\Big).
\end{equation}
Subtracting the two expressions in \eqref{eq:updates-19}, we obtain the exact recurrence relation for the perturbation
\begin{equation}\label{eq:deltaX-rec-19}
\Delta X_{j+1}
= (\Delta X_j)\Big(\sum_{k=0}^{18}\tilde{R}_j^{\,k}\Big)
  + X_j\Big(\sum_{k=0}^{18}(\tilde{R}_j^{\,k}-R_j^{\,k})\Big).
\end{equation}
To estimate the effect of the perturbation on the residual powers, 
we use the submultiplicative property of matrix norms.  
From \eqref{eq:pert-residual-19}, we have
\begin{equation}\label{eq:Phi-def-19}
\|\tilde{R}_j^k\| \le (\|R_j\|+\|A\|\,\|\Delta X_j\|)^k = \Phi_j^{\,k}.
\end{equation}
Furthermore, using the telescoping identity and retaining only the first-order terms in $\Delta X_j$, we obtain the bound
\begin{equation}\label{eq:bound-diff-19}
\|\tilde{R}_j^{\,k}-R_j^{\,k}\|
\le k\,\|R_j\|^{\,k-1}\,\|A\|\,\|\Delta X_j\|.
\end{equation}
Taking norms in \eqref{eq:deltaX-rec-19} and substituting 
\eqref{eq:Phi-def-19}–\eqref{eq:bound-diff-19}, while neglecting higher–order terms, yields
\[
\|\Delta X_{j+1}\|
\le \|\Delta X_j\|\sum_{k=0}^{18}\Phi_j^{\,k}
+ \|X_j\|\,\|A\|\,\|\Delta X_j\|\sum_{k=1}^{18}k\,\|R_j\|^{\,k-1}
+ \mathrm{O}(\|\Delta X_j\|).
\]
Using the crude but practical bounds
\[
\sum_{k=0}^{18}\Phi_j^{\,k}\le 19\max\{1,\|R_j\|^{18}\},\quad
\sum_{k=1}^{18}k\,\|R_j\|^{\,k-1}\le 18\max\{1,\|R_j\|^{18}\},
\]
we obtain the single–step perturbation bound
\begin{equation}\label{eq:final-step-19}
\|\Delta X_{j+1}\|
\le 19\,\|\Delta X_j\|\,\max\{1,\|R_j\|^{18}\}\,\big(1+18\,\|A\|\,\|X_j\|\big)
+ \mathrm{O}(\|\Delta X_j\|).
\end{equation}
Finally, iterating inequality \eqref{eq:final-step-19} from index \(i=0\) to \(i=j\) gives
\[
\|\Delta X_{j+1}\|
\le 19^{\,j+1}\Big(\prod_{i=0}^{j}\max\{1,\|R_i\|^{18}\}\,(1+18\,\|A\|\,\|X_i\|)\Big)\|\Delta X_0\|
+ \text{higher-order terms},
\]
which is the stated bound. This completes the proof.
\end{proof}

\subsection{Computational efficiency of the methods}
In this subsection, we provide theoretical insights into the computational efficiency of the proposed quaternion iterative schemes using the computational efficiency index (CEI).  
The CEI offers a practical measure for comparing iterative methods in terms of their convergence order relative to the computational effort required per iteration.  

The CEI is defined as
\[
\mathrm{CEI}=k^{1/\eta},
\]
where \(k\) denotes the local order of convergence and \(\eta\) represents the number of quaternion matrix–matrix multiplications performed per iteration (each quaternion multiplication being counted as a unit cost).  
A higher CEI value indicates a more efficient method, as it achieves a faster convergence rate for a given computational workload.

Using the multiplication counts derived in the previous subsections, the approximate CEI values for the algorithms considered are the following:
\[
\begin{array}{lcl}
\text{QNS} &:& \mathrm{CEI} = 2^{1/2} \approx 1.414,\\[2pt]
\text{QSAI}             &:& \mathrm{CEI} = 10^{1/6} \approx 1.467,\\[2pt]
\text{QHPI19}           &:& \mathrm{CEI} = 19^{1/7} \approx 1.522,\\[2pt]
\text{QRAPID}           &:& \mathrm{CEI} \approx 5^{1/6} \approx 1.307\;(N=0),\quad
                                         48^{1/18} \approx 1.239\;(N=5).
\end{array}
\]
From the CEI standpoint, higher values correspond to superior theoretical efficiency.  
Among the compared methods, the QHPI$19$ method achieves the highest CEI, indicating that it provides the best balance between convergence speed and computational cost.  
The QSAI method ranks second, maintaining a favorable compromise between high order and moderate multiplication count.  
The QNS method remains simple and numerically robust but exhibits a lower CEI because of its lower convergence order. Finally, the efficiency of the QRAPID method depends on the step parameter \(N\); for moderate values of \(N\), its CEI remains below those of QHPI$19$ and QSAI, although it can still be advantageous when rapid early approximations are desired.

\subsection{Preconditioning for quaternion linear systems}
Solving large-scale quaternion linear systems of the form
\[
A X = B, \text{~where~} A \in \mb{Q}^{n \times n}, ~ B \in \mb{Q}^{n \times m},
\]
is often challenging, especially when the coefficient matrix $A$ is ill-conditioned.   For these problems, iterative Krylov–subspace methods such as the global quaternion full orthogonalization method (Gl-QFOM) and the global quaternion GMRES (Gl-QGMRES)~\cite{MR4861347} are widely used. However, the convergence of these methods can be slow, resulting in high computational costs for large-scale systems.

Preconditioning is a standard technique to accelerate convergence.  
A preconditioner $M \in \mb{Q}^{n \times n}$ is an approximate inverse of $A$, which transforms the original system into
\[
M A X = M B.
\]
The preconditioned system is designed to have more favorable spectral properties, which typically reduces the number of iterations required by iterative solvers while maintaining low additional computational cost.  
An effective preconditioner should be inexpensive to construct and apply and should result in a system that is easier to solve than the original.

In this work, we propose to construct the preconditioner \(M\) using QSAI method introduced earlier. The QSAI algorithm explicitly computes a high-accuracy approximation to \(A^{-1}\) through a small number of quaternion matrix multiplications, providing an efficient preconditioner.  
Because the QSAI method is algebraically consistent with quaternion arithmetic, it integrates seamlessly with iterative solvers such as Gl–QFOM and Gl–QGMRES. The resulting QSAI-based preconditioned system
$M A X = M B$ offers faster convergence, improved numerical stability, and reduced overall computational cost compared to unpreconditioned  quaternion systems. This combination of high-order approximation and iterative refinement provides a robust framework for solving large-scale quaternion linear systems efficiently.

\section{Numerical Experiment}\label{sec4}
This section presents the numerical results that demonstrate the accuracy, efficiency and stability of the proposed quaternion iterative methods—QRAPID, QSAI, and QHPI$19$—for computing the Moore–Penrose inverse of quaternion matrices. All experiments were performed in \textsc{MATLAB} using double precision arithmetic on a workstation equipped with an Intel(R) Core(TM) $i9–12900K$ CPU ($3.2$~GHz), and $32$~GB of RAM. 

For a quaternion matrix \(A \in \mathbb{Q}^{m \times n}\) and its computed Moore–Penrose inverse \(X = A^\dagger \in \mathbb{Q}^{n \times m}\), the following error metrics are used to measure the accuracy of the computed solution:
\begin{equation}\label{eq_error}
    E_1 = \|AXA - A\|_F, ~ 
    E_2 = \|XAX - X\|_F, ~  
    E_3 = \|(AX)^H - AX\|_F,  ~ 
    E_4 = \|(XA)^H - XA\|_F.
\end{equation}
These error metrics respectively correspond to the four Penrose conditions and collectively quantify how closely the computed matrix \(X\) satisfies the defining properties of the Moore–Penrose inverse.  
The initial approximation is set to \(X_0 = \alpha A^T\) with \(\alpha = 1 / \|A\|_F^2\), and all iterations use a stopping tolerance of \(10^{-10}\) to ensure numerical precision. We first present a low-dimensional example to illustrate the behavior and accuracy of the proposed quaternion iterative algorithms, followed by experiments in larger-scale settings.

\begin{example}\label{examp1}
Consider the quaternion matrix \(A \in \mathbb{Q}^{3\times 3}\):
\[
A =
\begin{bmatrix}
6+3\mathbf{i}+5\mathbf{j}+2\mathbf{k} & 1+5\mathbf{i}+2\mathbf{j}+3\mathbf{k} & \mathbf{i}+7\mathbf{j}+8\mathbf{k}\\
2+\mathbf{i}+\mathbf{j}+\mathbf{k} & 3+3\mathbf{i}+\mathbf{j}+\mathbf{k} & 2+5\mathbf{i}+2\mathbf{j}+\mathbf{k}\\
4+2\mathbf{i}+2\mathbf{j}+2\mathbf{k} & 6+6\mathbf{i}+2\mathbf{j}+2\mathbf{k} & 4+10\mathbf{i}+4\mathbf{j}+2\mathbf{k}
\end{bmatrix}.
\]
The scaling parameter \(\alpha\) was determined using the complex representation approach, which yields \(\alpha = 2.058856\times10^{-3}\).  
Starting from \(X_0 = \alpha A^T\), the Moore–Penrose inverse \(A^\dagger\) was computed using the three proposed iterative schemes: QSAI, QRAPID, and QHPI$19$.  
All methods converged to identical results (within displayed precision), verifying the consistency of the algorithms. The computed quaternion Moore–Penrose inverse is
\begin{align*}
A^\dagger(:,1) &=
\begin{bmatrix}
0.0627 -0.0325\mathbf{i}-0.0520\mathbf{j}+0.0236\mathbf{k}\\
-0.0118-0.0229\mathbf{i}+0.0102\mathbf{j}+0.0314\mathbf{k}\\
-0.0042+0.0458\mathbf{i}-0.0116\mathbf{j}-0.0362\mathbf{k}
\end{bmatrix},\\
A^\dagger(:,2) &=
\begin{bmatrix}
-0.0028+0.0085\mathbf{i}+0.0051\mathbf{j}-0.0264\mathbf{k}\\
0.0164-0.0075\mathbf{i}-0.0129\mathbf{j}-0.0092\mathbf{k}\\
0.0045-0.0225\mathbf{i}+0.0071\mathbf{j}+0.0081\mathbf{k}
\end{bmatrix},\\
A^\dagger(:,3) &=
\begin{bmatrix}
-0.0055+0.0170\mathbf{i}+0.0102\mathbf{j}-0.0527\mathbf{k}\\
0.0327-0.0150\mathbf{i}-0.0259\mathbf{j}-0.0183\mathbf{k}\\
0.0091-0.0449\mathbf{i}+0.0142\mathbf{j}+0.0163\mathbf{k}
\end{bmatrix}.
\end{align*}
To assess accuracy and convergence, the error metrics \(E_1\)–\(E_4\) from~\eqref{eq_error} were calculated for each method along with the total number of iterations required for convergence.  
The results are summarized in Table~\ref{tab:example1_errors}.
\begin{table}[h!]
\centering
\begin{tabular}{lccccc}
\toprule
Method & Iterations & $E_1$ & $E_2$ & $E_3$ & $E_4$ \\
\midrule
QSAI        & $4$ & $3.84\times10^{-15}$ & $5.06\times10^{-16}$ & $1.14\times10^{-15}$ & $5.22\times10^{-16}$ \\[2pt]
QRAPID      & $4$ & $3.02\times10^{-15}$ & $8.06\times10^{-15}$ & $4.46\times10^{-16}$ & $4.96\times10^{-16}$ \\[2pt]
QHPI$19$      & $3$ & $1.29\times10^{-14}$ & $2.28\times10^{-15}$ & $2.95\times10^{-15}$ & $1.80\times10^{-15}$ \\
\bottomrule
\end{tabular}
\caption{Comparison of QSAI, QRAPID, and QHPI$19$ for Example~\ref{examp1}.  Error metrics \(E_1\)–\(E_4\) correspond to the four Penrose conditions defined in~\eqref{eq_error}.}
\label{tab:example1_errors}
\end{table}
All three methods achieve residuals below \(10^{-14}\), confirming that the Moore–Penrose conditions are satisfied to machine precision. Among them, QHPI$19$ achieves the fastest convergence in just three iterations, reflecting its higher theoretical order, while QSAI and QRAPID maintain a comparable accuracy with excellent numerical stability. 
\end{example}

After verifying the correctness and convergence of the proposed quaternion iterative methods on a small-scale example, we next examine their behavior on larger and structurally diverse systems.  
This set of experiments aims to evaluate the scalability, computational efficiency, and robustness of the methods under different matrix characteristics—square, rectangular, and rank-deficient cases.  
For comparison, we also include results from the classical QNS iteration and the exact QSVD-based pseudoinverse, which serve as established benchmarks for accuracy and stability.
\begin{example}
In this example, we compare the performance of the proposed methods—QSAI, QRAPID, and QHPI$19$—against the baseline approaches, namely the QNS iteration and the QSVD-based pseudoinverse. Each method is tested on quaternion matrices of various sizes and structures to assess the convergence rate, computational cost, and numerical accuracy.  
The experiments are grouped into three categories: $(i)$ square systems, $(ii)$ rectangular systems (both overdetermined and underdetermined), and $(iii)$ rank-deficient systems.

\noindent\textbf{Square matrices.}  
For square matrices of order \(n = 300\) and \(n = 1500\), the results presented in Table~\ref{tab:300_1500_results} show that all three hyperpower-type methods—QSAI, QRAPID, and QHPI$19$—exhibit rapid convergence within $10-15$ iterations.  
They maintain high numerical accuracy, with residual errors consistently below \(10^{-10}\)–\(10^{-12}\), while requiring substantially less computation time than the QSVD approach.  
Although the QNS iteration converges in all cases, it demands over $80$ iterations and produces larger residuals, especially in the consistency metric \(\|XAX - X\|_F\), indicating slower and less stable convergence.

\begin{table}[hbt!]\centering
\begin{tabular}{lcccccc}
\hline
Method & Time(s) & Iterations 
& $E_1$ 
& $E_2$ 
& $E_3$ 
& $E_4$ \\
\hline
\multicolumn{7}{c}{Matrix size $n=300$}\\
\hline
QSAI            & $7.308e-01$ & $11$ & $1.79e-11$ & $1.37e-12$ & $1.11e-12$ & $4.00e-11$ \\
QRAPID          & $1.217e+00$ & $7$  & $1.03e-11$ & $1.40e-12$ & $7.17e-13$ & $2.69e-11$ \\
QHPI$19$          & $7.797e-01$ & $9$  & $1.39e-11$ & $1.38e-12$ & $9.13e-13$ & $2.99e-11$ \\
QSVD           & $8.378e+00$ & $1$  & $7.48e-12$ & $4.47e-12$ & $3.02e-12$ & $3.12e-12$ \\
QNS & $2.433e+00$ & $81$ & $8.27e-12$ & $3.51e-09$ & $2.27e-11$ & $5.10e-13$ \\

\hline
\multicolumn{7}{c}{Matrix size $n=1500$}\\
\hline
QSAI            & $3.719e+01$ & $13$  & $1.41e-10$ & $6.04e-12$ & $5.06e-12$ & $2.50e-10$ \\
QRAPID          & $8.478e+01$ & $8$   & $1.10e-10$ & $6.04e-12$ & $3.43e-12$ & $1.96e-10$ \\
QHPI$19$          & $4.502e+01$ & $11$  & $1.26e-10$ & $5.94e-12$ & $4.29e-12$ & $2.95e-10$ \\
QSVD           & $4.024e+02$ & $1$   & $7.88e-11$ & $2.01e-11$ & $1.88e-11$ & $2.12e-11$ \\
QNS & $1.256e+02$ & $92$  & $9.14e-11$ & $4.87e-09$ & $1.60e-10$ & $3.21e-12$ \\

\hline
\end{tabular}
\caption{Comparison of quaternion pseudoinverse methods for square matrices of size \(n=300\) and \(n=1500\).}
\label{tab:300_1500_results}
\end{table}
\noindent\textbf{Rectangular and rank-deficient systems.}  
The results for rectangular and rank-deficient matrices, summarized in Table~\ref{tab:three_cases_results}, further confirm the robustness of the proposed algorithms.  
For overdetermined (\(1000\times500\)) and underdetermined (\(500\times1000\)) systems, QSAI and QRAPID achieve accuracies on the order of \(10^{-12}\) in only $5-7$ iterations, while QHPI$19$ provides comparable precision with slightly higher computational cost.  
In the rank-deficient case (\(1000\times1000\)), all three hyperpower-type methods maintain similar error levels and remain approximately an order of magnitude faster than QSVD. As before, QNS converges slowly and yields larger residuals.

\begin{table}[hbt!]\centering

\begin{tabular}{lcccccc}
\hline
Method & Time(s) & Iterations 
& $E_1$ 
& $E_2$ 
& $E_3$ 
& $E_4$ \\
\hline
\multicolumn{7}{c}{Overdetermined ($1000\times500$)}\\
\hline
QSAI            & $4.408e+00$ & $7$   & $2.55e-12$ & $3.50e-15$ & $4.34e-13$ & $1.59e-13$ \\
QRAPID          & $8.293e+00$ & $5$   & $9.19e-13$ & $1.90e-15$ & $6.94e-14$ & $6.15e-14$ \\
QHPI$19$          & $5.657e+00$ & $6$   & $1.15e-11$ & $9.47e-15$ & $9.38e-13$ & $4.68e-13$ \\
QSVD           & $4.297e+01$ & $1$   & $2.92e-12$ & $1.15e-14$ & $2.83e-13$ & $1.91e-13$ \\
QNS & $7.044e+00$ & $60$  & $4.65e-10$ & $9.85e-12$ & $6.43e-14$ & $3.21e-14$ \\

\hline
\multicolumn{7}{c}{Underdetermined ($500\times1000$)}\\
\hline
QSAI            & $1.844e+00$ & $7$   & $2.30e-12$ & $3.55e-15$ & $5.59e-14$ & $1.40e-13$ \\
QRAPID          & $4.089e+00$ & $5$   & $8.99e-13$ & $2.16e-15$ & $3.49e-14$ & $9.84e-14$ \\
QHPI$19$          & $1.752e+00$ & $6$   & $1.01e-12$ & $2.17e-15$ & $3.80e-14$ & $1.07e-13$ \\
QSVD           & $4.239e+01$ & $1$   & $2.69e-12$ & $7.95e-15$ & $1.81e-13$ & $2.39e-13$ \\
QNS & $7.126e+00$ & $60$  & $4.69e-10$ & $1.01e-11$ & $3.21e-14$ & $6.60e-14$ \\

\hline
\multicolumn{7}{c}{Rank-deficient ($1000\times1000$)}\\
\hline
QSAI            & $8.789e+00$ & $10$  & $1.37e-09$ & $3.60e-14$ & $1.64e-12$ & $4.23e-11$ \\
QRAPID          & $1.678e+01$ & $6$   & $1.07e-09$ & $3.64e-14$ & $1.29e-12$ & $4.28e-11$ \\
QHPI$19$          & $8.570e+00$ & $8$   & $1.33e-09$ & $3.58e-14$ & $1.59e-12$ & $4.24e-11$ \\
QSVD           & $1.123e+02$ & $1$   & $8.15e-10$ & $2.11e-13$ & $8.52e-12$ & $9.13e-12$ \\
QNS & $3.273e+01$ & $73$  & $8.88e-10$ & $4.48e-11$ & $2.46e-11$ & $1.06e-12$ \\
\hline
\end{tabular}
\caption{Comparison of quaternion pseudoinverse methods for three structural test cases: overdetermined (\(1000\times500\)), underdetermined (\(500\times1000\)), and rank-deficient (\(1000\times1000\)).}
\label{tab:three_cases_results}
\end{table}

\medskip
\noindent\textbf{Scalability and efficiency.} 
The scalability trends are illustrated in Figure~\ref{fig:SAI_results}, which compares CPU time and residual errors for increasing matrix sizes. As \(n\) grows, the QSVD computation becomes prohibitively expensive, whereas the proposed iterative schemes scale efficiently and maintain accuracy near machine precision. Among them, QSAI and QHPI$19$ exhibit the best accuracy–efficiency balance, while QNS remains less effective for large-scale systems.

\begin{figure}[hbt!]
\centering
\includegraphics[width=0.49\textwidth]{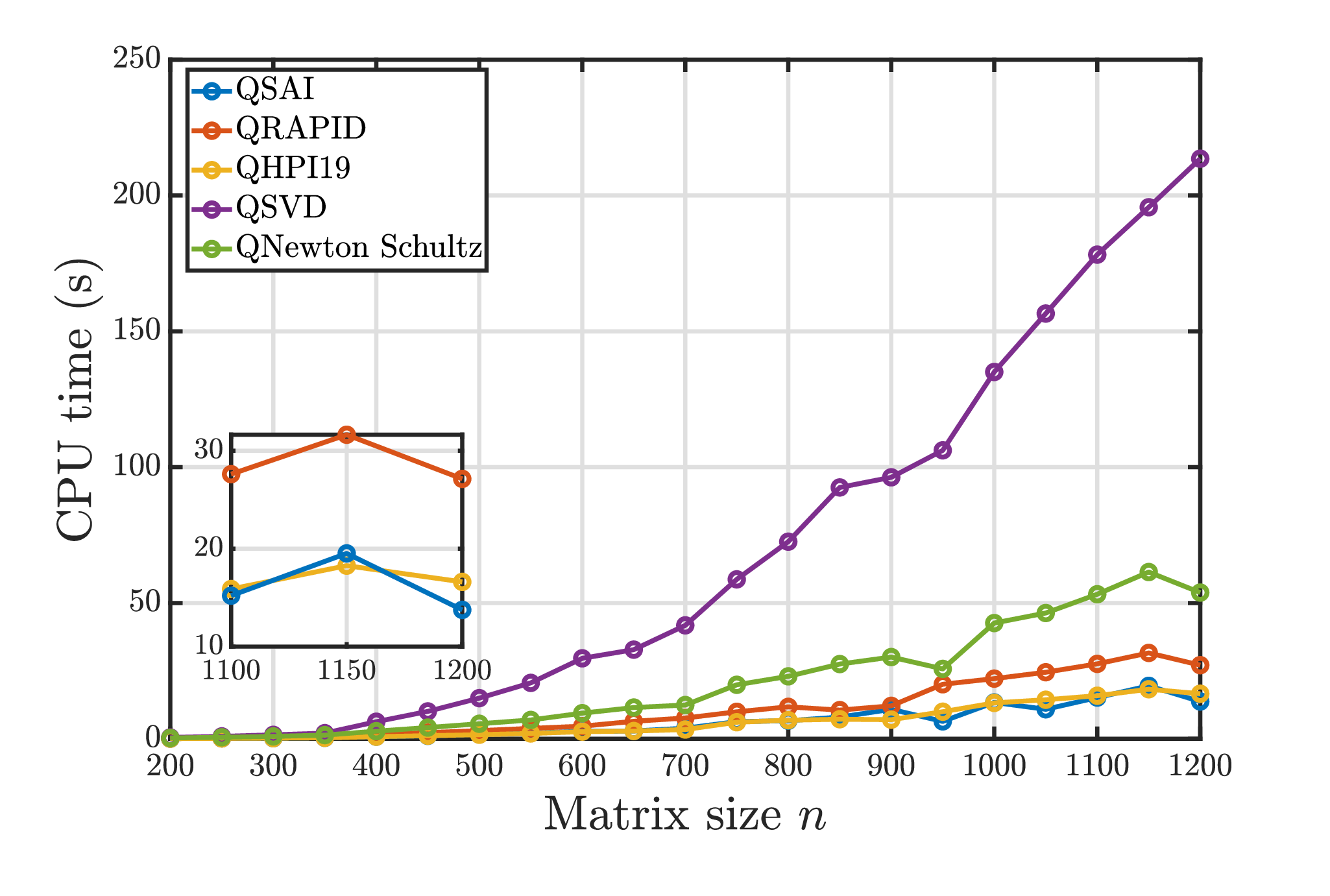}

\vspace{4pt}

\parbox{0.49\textwidth}{\centering  (a) CPU Time}

\vspace{1ex}

\includegraphics[width=0.49\textwidth]{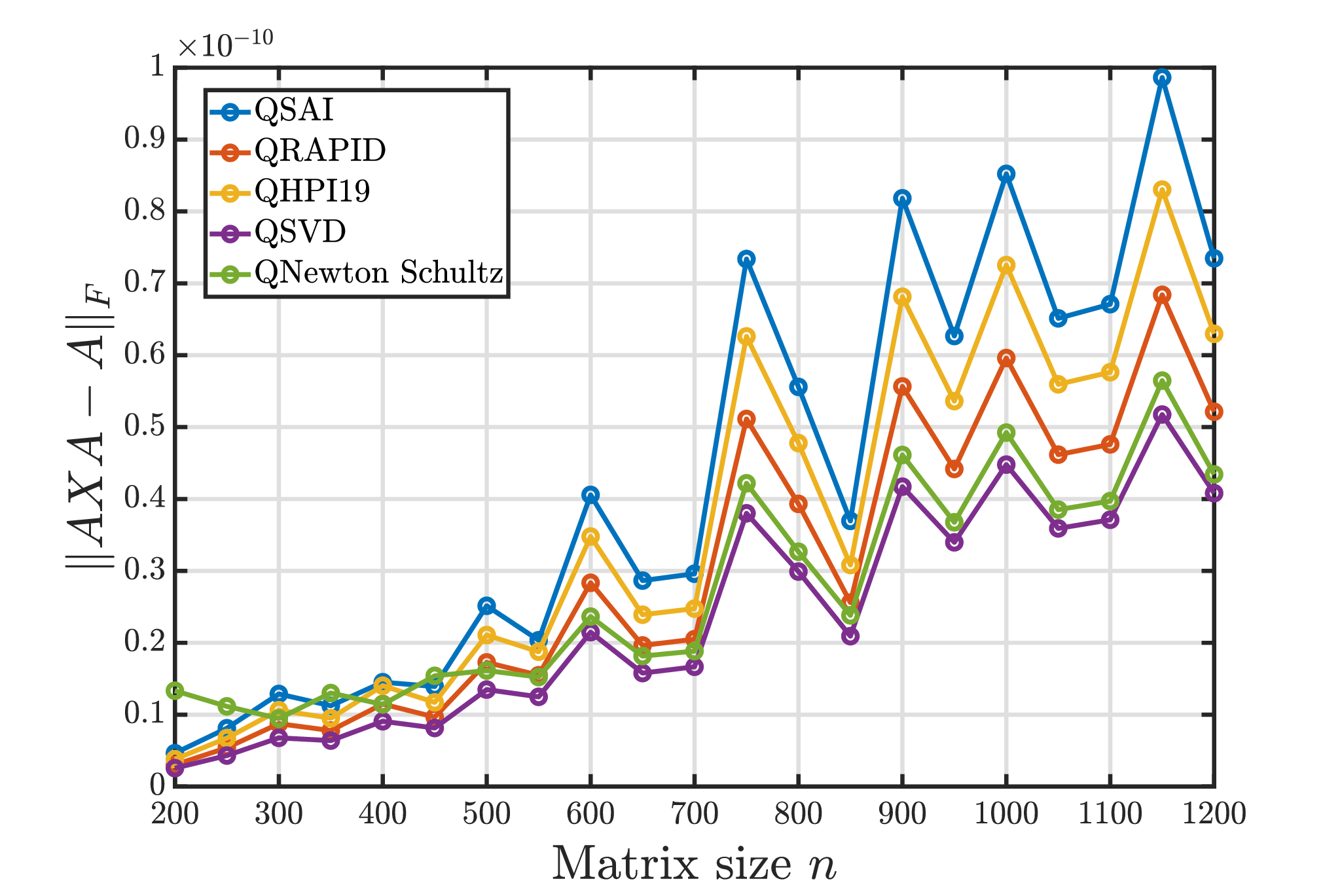}\hfill
\includegraphics[width=0.49\textwidth]{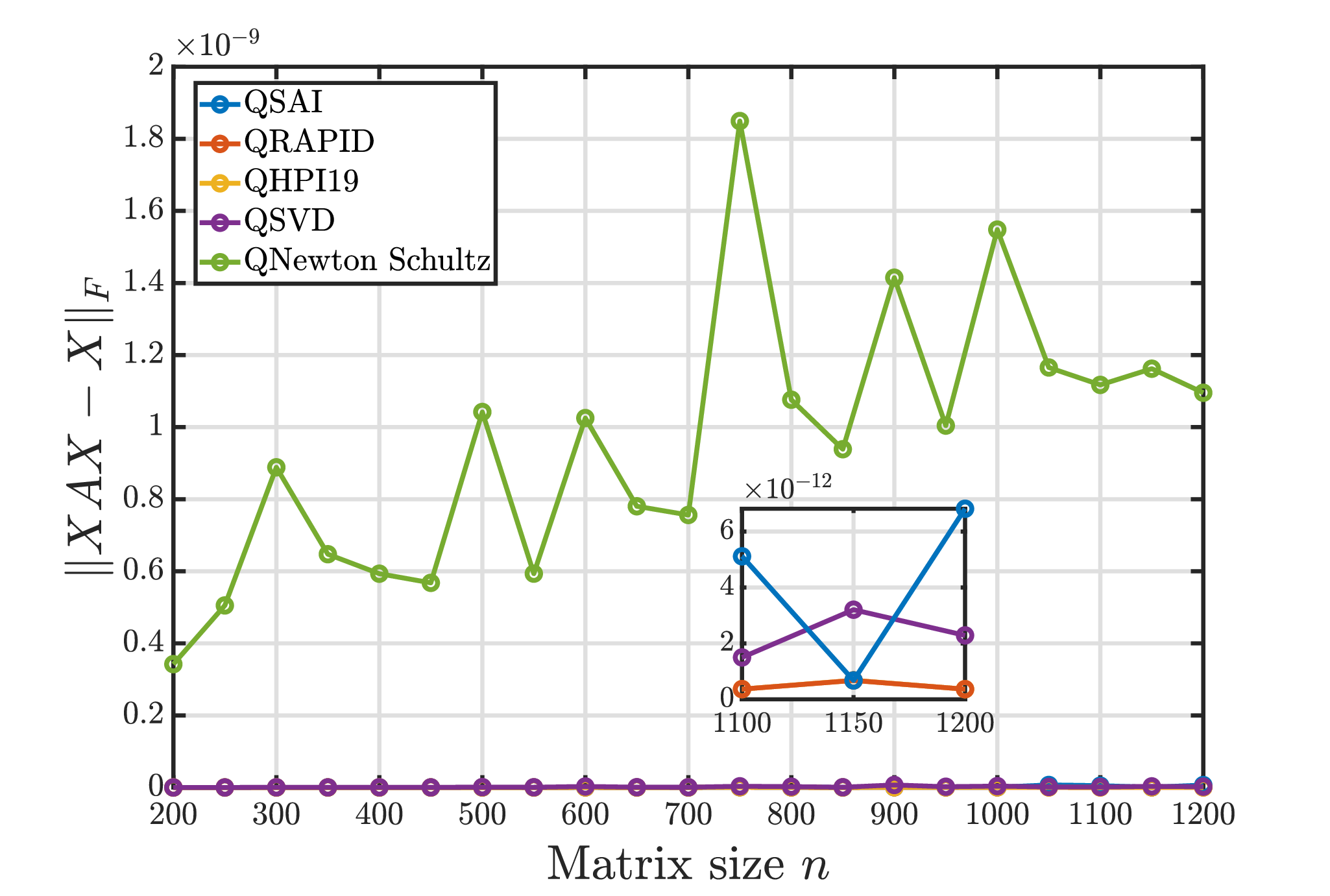}

\vspace{4pt}

\parbox{0.49\textwidth}{\centering  (b) $E_1$}\hfill
\parbox{0.49\textwidth}{\centering  (c) $E_2$}

\vspace{1ex}

\includegraphics[width=0.49\textwidth]{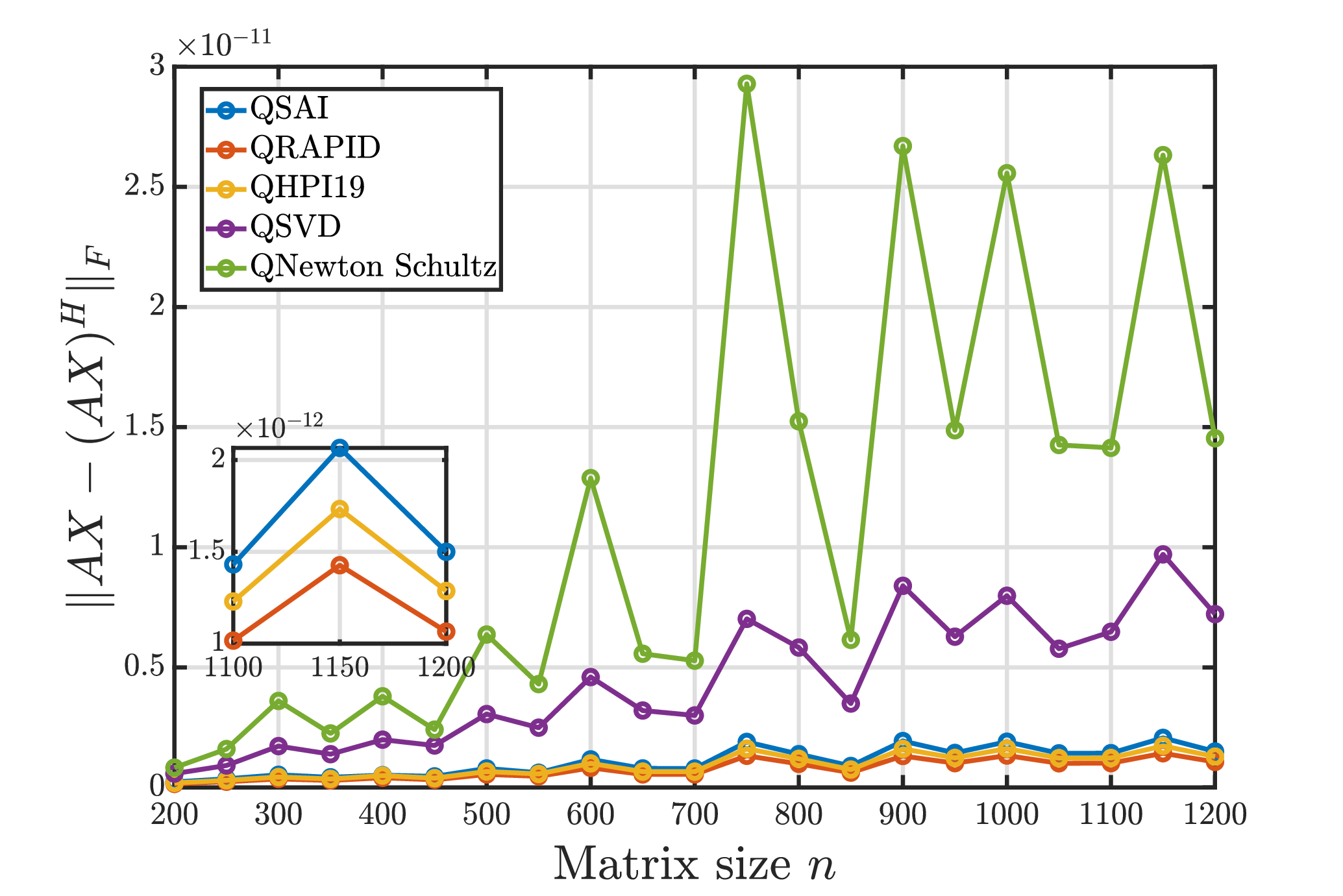}\hfill
\includegraphics[width=0.49\textwidth]{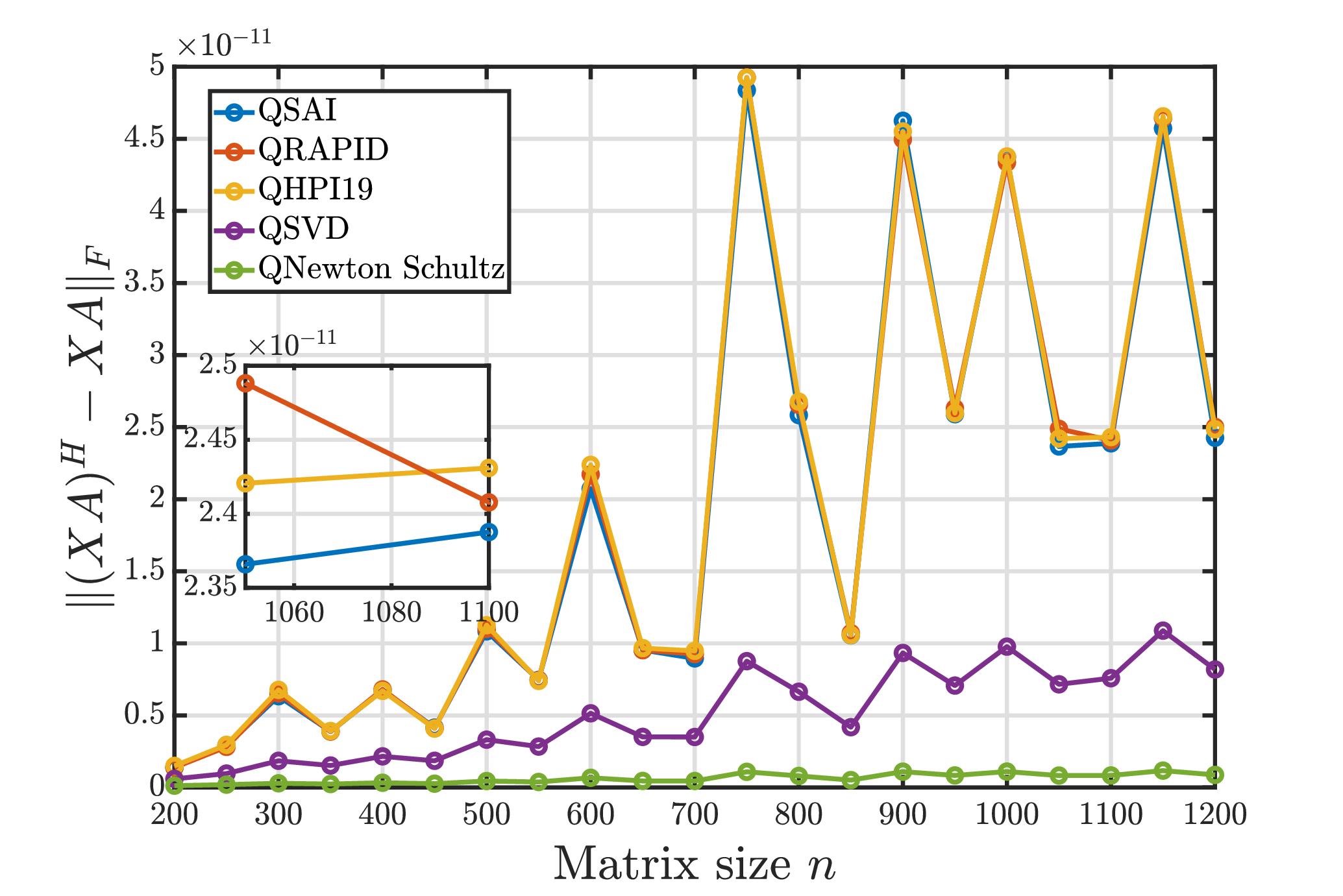}

\vspace{4pt}

\parbox{0.49\textwidth}{\centering  (d) $E_3$}\hfill
\parbox{0.49\textwidth}{\centering  (e) $E_4$}

\caption{Performance comparison of quaternion pseudoinverse methods: CPU time and error norms across matrix sizes.}
\label{fig:SAI_results}
\end{figure}

The comprehensive numerical evidence highlights the superior computational efficiency and stability of the proposed quaternion hyperpower-based methods. Among them, QSAI consistently delivers the most robust performance across all test configurations, while QRAPID and QHPI$19$ achieve comparable accuracy with slightly different trade-offs in iteration count and runtime. Overall, these results confirm that quaternion-tailored hyperpower algorithms provide a powerful and scalable alternative to QSVD and QNS for large-scale quaternion inverse problems.
\end{example}
Having demonstrated the accuracy, scalability, and robustness of the proposed quaternion iterative schemes across various problem sizes and structures, we now examine one of their key computational advantages—the impact of factorization on efficiency. Although the factorized and unfactorized hyperpower iterations share the same theoretical convergence order, the factorized forms are designed to minimize redundant quaternion matrix–matrix multiplications, thereby offering substantial savings in runtime. The following example quantitatively illustrates this improvement.

\begin{example}\label{exam_hyper}
Consider the quaternion matrix inverse approximation problem
\[
A \in \mathbb{Q}^{n\times n}, \quad X \approx A^\dagger,
\]
where each entry of \(A\) is generated from independent Gaussian distributions in all four
quaternion components. The goal is to approximate the Moore–Penrose inverse \(A^\dagger\) using different hyperpower-based iterative schemes.

\noindent\textbf{Experimental setup.}
To ensure a consistent comparison, all algorithms employ the same initialization parameter \(\alpha = 7.8798\times10^{-5}\).  
Each iteration terminates when \(\|X_{k+1} - X_k\|_F \le 10^{-10}\) or when the iteration count reaches $500$. The following four solvers are evaluated:
\begin{itemize}
  \item \textbf{QHPI$19$:} proposed $19$th–order factorized hyperpower iteration,
  \item \textbf{QSAI:} proposed $10$th–order factorized strong approximate inverse,
  \item \textbf{QHONSchultz($19$)}~\cite{leplat2025iterative}: unfactorized hyperpower iteration of order $19$,
  \item \textbf{QHONSchultz($10$)}~\cite{leplat2025iterative}: unfactorized hyperpower iteration of order $10$.
\end{itemize}
For each solver, we record the CPU time, iteration count, and the four error metrics \(E_1\)–\(E_4\) defined in~\eqref{eq_error}, which correspond to the Penrose conditions and collectively assess the numerical accuracy of the computed Moore–Penrose inverse.

\noindent\textbf{Results and discussion.}
Table~\ref{tab:hyper_n800} summarizes the results for a representative case with \(n=800\). Both proposed factorized methods QHPI$19$ and QSAI achieve the same accuracy as the unfactorized QHONSchultz algorithms, with all residuals in the range \(10^{-11}\)–\(10^{-13}\). However, the factorized versions complete in roughly half the CPU time, confirming
that factorization substantially reduces redundant quaternion multiplications without affecting numerical precision. 

\begin{table}[hbt!]
\centering
\begin{tabular}{lcccccc}
\toprule
Method & Time (s) & Iters & $E_1$ & $E_2$ & $E_3$ & $E_4$ \\
\midrule
QHPI$19$ & $6.188$ & $8$  & $6.27\times 10^{-11}$ & $8.06\times 10^{-13}$ & $1.57\times 10^{-12}$ & $4.66\times 10^{-11}$  \\
QSAI & $6.594$ & $10$ & $7.38\times 10^{-11}$ & $8.14\times 10^{-13}$ & $1.85\times 10^{-12}$ & $4.53\times 10^{-11}$  \\
QHONSchultz(order $19$) & $13.87$ & $8$  & $6.28\times 10^{-11}$ & $8.02\times 10^{-13}$ & $1.57\times 10^{-12}$ & $4.53\times 10^{-11}$  \\
QHONSchultz(order $10$)   & $11.05$ & $10$ & $6.28\times 10^{-11}$ & $8.05\times 10^{-13}$ & $1.57\times 10^{-12}$ & $4.39\times 10^{-11}$  \\
\bottomrule
\end{tabular}
\caption{Comparison for \(n=800\): factorized versus unfactorized hyperpower iterations.}
\label{tab:hyper_n800}
\end{table}
Figure~\ref{fig:cpu_time} shows the CPU time scaling with matrix dimension. The growth is approximately quadratic in \(n\), consistent with the \(\mathrm{O}(n^3)\) cost of quaternion matrix–matrix multiplications. Nevertheless, the unfactorized QHONSchultz methods are markedly slower. For instance, at \(n=800\), QHONSchultz($19$) requires more than twice the CPU time of QHPI$19$, despite achieving nearly identical accuracy (see Table~\ref{tab:hyper_n800}).  

\begin{figure}[hbt!]
\centering
\includegraphics[width=0.55\textwidth]{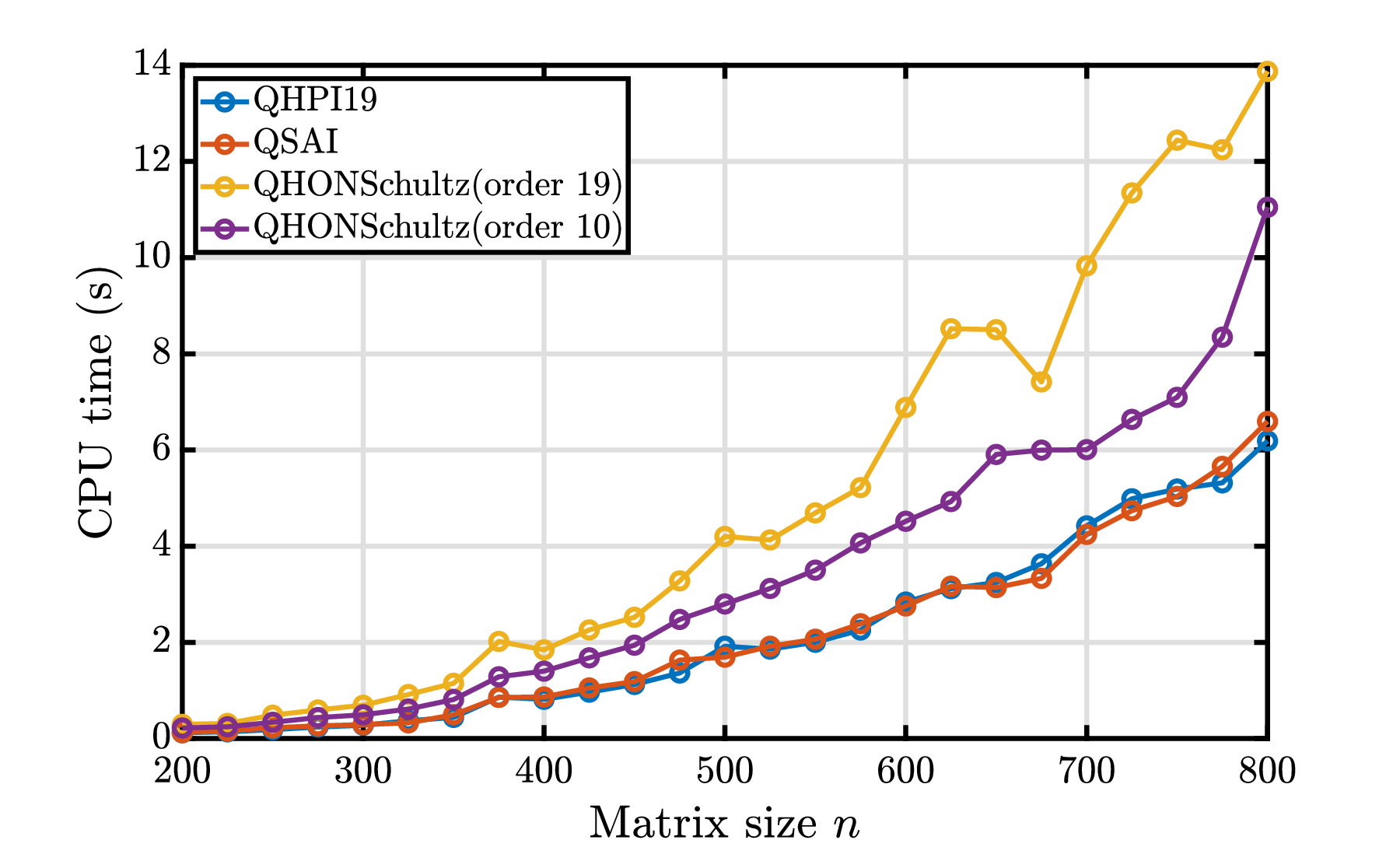}
\caption{CPU time versus matrix size for factorized and unfactorized hyperpower iterations.}
\label{fig:cpu_time}
\end{figure}

All methods yield diagnostic errors on the order of \(10^{-11}\)--\(10^{-13}\), demonstrating that factorization does not compromise numerical precision. Rather, it eliminates redundant polynomial multiplications inherent in the unfactorized updates. Specifically, while the unfactorized QHONSchultz iterations have per-iteration complexity \(\mathrm{O}(p n^3)\) (with \(p=10\) or \(19\)),
the factorized schemes reorganize the polynomial computation such that each iteration requires only \(\mathrm{O}(n^3)\) operations, independent of \(p\). This structural simplification explains the pronounced runtime advantage observed in Figure~\ref{fig:cpu_time} and establishes the practicality of the proposed factorized framework for large-scale quaternion computations. In summary, the factorized hyperpower iterations achieve the same high-order accuracy as the classical unfactorized schemes while substantially reducing computational cost. 
\end{example}
Following the analysis of factorized hyperpower schemes, we now evaluate the effectiveness of the proposed QSAI method when used as a preconditioner within Krylov subspace algorithms for solving large-scale quaternion linear systems. Preconditioning is particularly crucial in such settings, as it can significantly improve convergence speed and overall computational efficiency. Here, we focus on two generalized quaternion Krylov solvers Gl-QFOM and Gl-QGMRES implemented following Algorithms~$4$ and~$5$ of~\cite{MR4861347}, and investigate the acceleration achieved by incorporating the QSAI preconditioner.
in~\cite{MR4861347}.  
\begin{example}\label{exam1}
We consider the quaternion linear system $A X = B$, where the coefficient matrix is given by
\[
A = A^{(s)} + A^{(x)} \mathbf{i} + A^{(y)} \mathbf{j} + A^{(z)} \mathbf{k}, \quad \text{where~}
A^{(x)} = -A^{(s)}, ~ A^{(y)} = 2A^{(s)}, ~ A^{(z)} = 1.5A^{(s)},
\]
with $A^{(s)} \in \mathbb{R}^{238 \times 238}$ chosen as the \texttt{saylr1} matrix from the Matrix Market collection\footnote{See \url{https://math.nist.gov/MatrixMarket/}.}.  
The right-hand side matrix is generated as
\[
B = B^{(s)} + B^{(x)} \mathbf{i} + B^{(y)} \mathbf{j} + B^{(z)} \mathbf{k},
\]
where each $B^{(\ell)} \in \mathbb{R}^{238 \times m}$ ($\ell \in \{s,x,y,z\}$) is generated with uniformly distributed random entries.
 The number of columns $m$ in $B$ is varied as $m \in \{3,6,9,12\}$.  

\medskip
\noindent\textbf{Preconditioning setup.}
Both the Gl-QFOM and Gl-QGMRES are applied to the system in their standard and preconditioned forms. For the preconditioned variants, we employ the QSAI preconditioner \(M\), transforming the system into
\[
M A X = M B.
\]
\noindent\textbf{Evaluation metrics.}  
For each test, we record the number of iterations (IT), CPU time in seconds (Time), and the final relative residual (RR).  
The relative residual after the \(j\)-th iteration is computed as
\[
RR_j = \frac{\| B - A X_j \|_F}{\| B - A X_0 \|_F},
\]
where \(X_j\) denotes the current approximate solution and \(X_0\) is the initial guess.  The iteration process terminates when \(RR_j \le 10^{-6}\) or when the maximum iteration count \(k_{\max} = 3000\) is reached. All experiments are initialized with \(X_0 = 0\), and cases that fail to converge within \(k_{\max}\) iterations are marked with the symbol \(^{\dagger}\). For the preconditioned solvers, the reported CPU time includes both the preconditioner construction and the iterative solve.

\medskip
\noindent\textbf{Results and discussion.}  
Figure~\ref{fig_prec} depicts the convergence histories of Gl-QFOM, Gl-QGMRES, and their preconditioned counterparts for various values of \(m\). It is evident that the QSAI preconditioner substantially accelerates convergence, particularly as the system dimension increases. A detailed quantitative comparison is provided in Table~\ref{tab:quat_results}, which reports iteration counts, total CPU times, and final relative residuals. Across all test cases, the preconditioned methods converge in significantly fewer iterations and exhibit notable reductions in computational time, while maintaining residuals close to the prescribed tolerance. 

\begin{table}[hbt!]
\centering
\begin{tabular}{ccccc}
\toprule
$m$ & Algorithm &  IT & Time(s) & RR \\
\midrule
 $3$& Gl-QFOM \cite{MR4861347}  & $697$ & $10.9431$ & $6.6738e-07$ \\
& Gl-QGMRES \cite{MR4861347}&  $697$ & $10.7126$ & $6.0361e-07$ \\
& Preconditioned Gl-QFOM  & $308$ & $2.0735$ & $9.2945e-07$ \\
& Preconditioned Gl-QGMRES  & $244$ & $1.2898$ & $9.8907e-07$ \\
\midrule
$6$ & Gl-QFOM \cite{MR4861347}& $1268$ & $50.5642$ & $6.9230e-07$ \\
&Gl-QGMRES \cite{MR4861347} & $1264$ & $49.9777$ & $9.9840e-07$ \\
 &Preconditioned Gl-QFOM  & $542$ & $8.7815$ & $7.9485e-07$ \\
&Preconditioned Gl-QGMRES  & $419$ & $5.1615$ & $9.8399e-07$ \\
\midrule
$9$ & Gl-QFOM \cite{MR4861347} &  $1756$ & $278.1528$ & $9.8520e-07$ \\
   &  Gl-QGMRES \cite{MR4861347}&  $1752$ & $140.8921$ & $8.9032e-07$ \\
  & Preconditioned Gl-QFOM  & $746$ & $24.0042$ & $8.9876e-07$ \\
   & Preconditioned Gl-QGMRES  & $568$ & $13.2703$ & $9.9974e-07$ \\
\midrule
 $12$& Gl-QFOM \cite{MR4861347} & $2218$ & $449.8252$ & $9.5874e-07$ \\
& Gl-QGMRES \cite{MR4861347} &  $2212$ & $265.0176$ & $9.1120e-07$ \\
& Preconditioned Gl-QFOM  & $951$ & $47.6353$ & $9.4756e-07$ \\
& Preconditioned Gl-QGMRES  & $703$ & $23.9888$ & $9.9987e-07$ \\
\bottomrule
\end{tabular}
\caption{Numerical results of Example~\ref{exam1}.}
\label{tab:quat_results}
\end{table}

\begin{figure}[hbt!]
  \centering
  \includegraphics[width=0.48\textwidth]{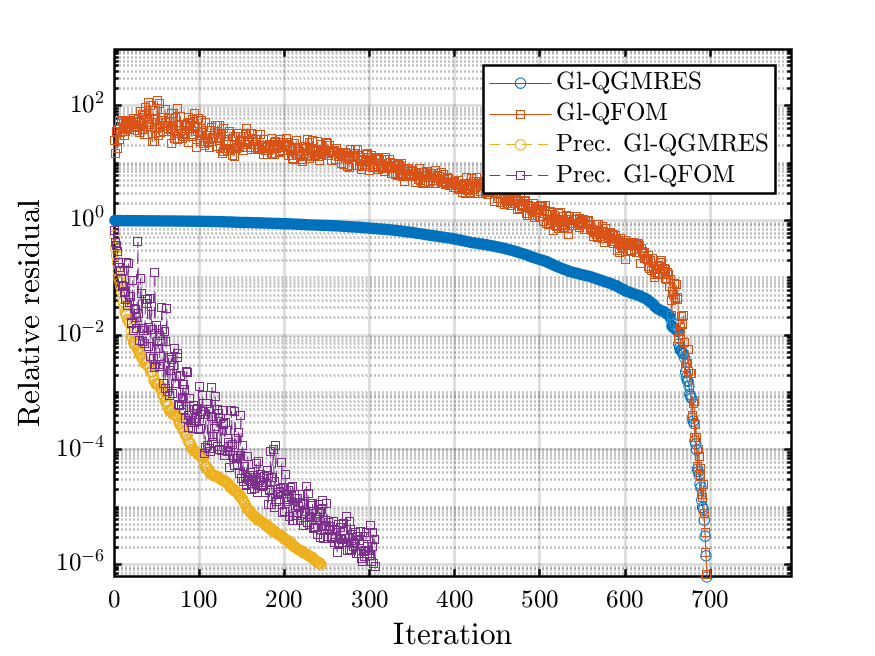}\hfill
  \includegraphics[width=0.48\textwidth]{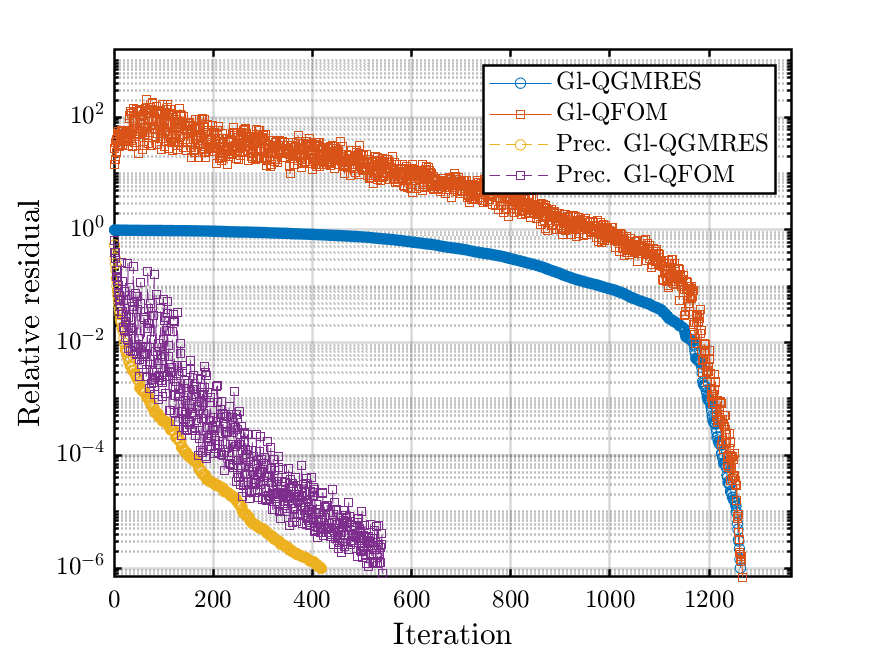}
  
  \vspace{4pt}
  
  \parbox{0.48\textwidth}{\centering  (a) Convergence curves for $m=3$.}\hfill
  \parbox{0.48\textwidth}{\centering  (b) Convergence curves for $m=6$.}
  
  \vspace{8pt}
  
  \includegraphics[width=0.48\textwidth]{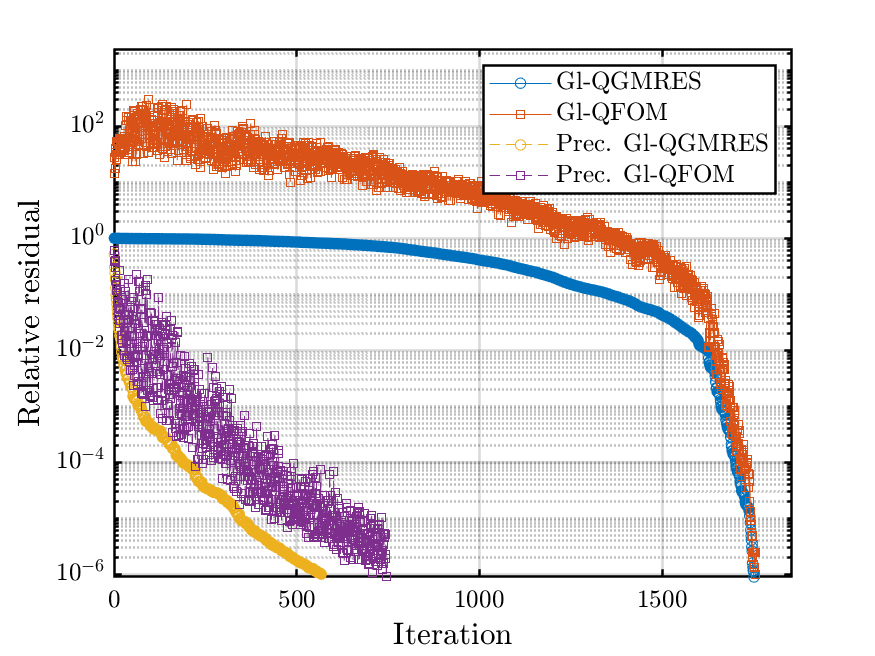}\hfill
  \includegraphics[width=0.48\textwidth]{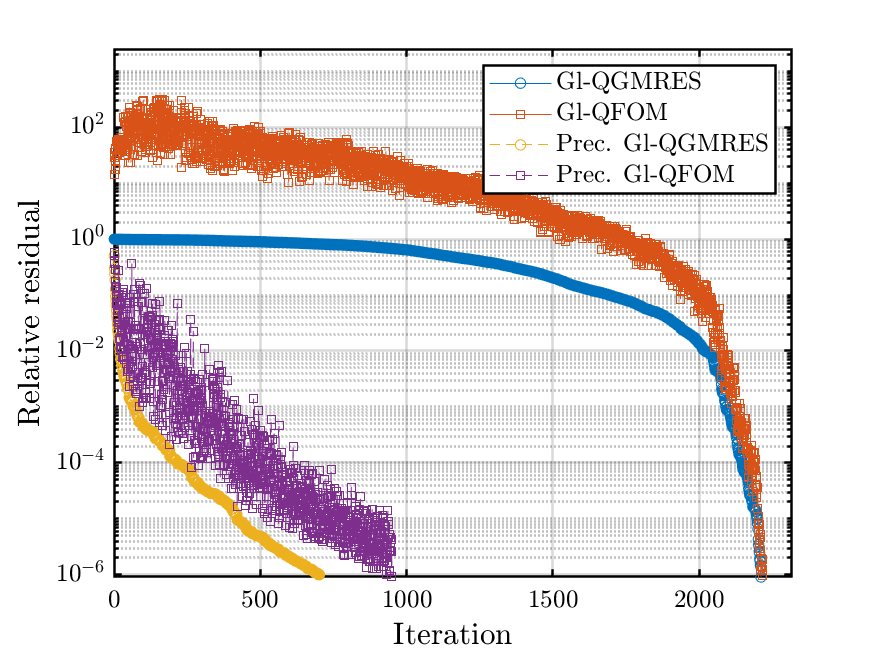}
  
  \vspace{4pt}
  
  \parbox{0.48\textwidth}{\centering  (c) Convergence curves for $m=9$.}\hfill
  \parbox{0.48\textwidth}{\centering  (d) Convergence curves for $m=12$.}
  
  \caption{Convergence curves of Gl-QFOM, Gl-QGMRES, and their preconditioned (prec.) versions 
for Example~\ref{exam1} with different values of $m$.}
  \label{fig_prec}
\end{figure}

\end{example}

\section{Applications}\label{sec5}
The proposed quaternion iterative methods are now applied to two representative problems to demonstrate their practical effectiveness and versatility. The first application focuses on color image completion using a CUR decomposition framework, while the second addresses the filtering of chaotic three-dimensional signals. Together, these applications illustrate how the proposed algorithms can efficiently handle multidimensional real-world data.

\subsection{Image completion via CUR}
Color images can be naturally modeled using quaternion algebra, which enables unified processing of the red, green, and blue channels. In this representation, each pixel is expressed as a purely imaginary quaternion, where the three imaginary components correspond to the RGB values and the scalar part is zero. This formulation captures inter-channel correlations.

In quaternion image completion, the objective is to reconstruct a missing or corrupted quaternion-valued matrix \( A \in \mathbb{Q}^{m \times n} \) from partially observed data.  A binary mask \( \Omega \in \{0,1\}^{m \times n} \) indicates the observed entries, while the matrix \( M \) stores the known pixel values. The standard iterative impute--reconstruct framework~\cite{wu2025efficient} can be written as
\begin{equation}
\left.
\begin{aligned}
X^{(t)} &\leftarrow \mathcal{L}\!\left(C^{(t)}\right), \\
C^{(t+1)} &\leftarrow \Omega \odot M + (1-\Omega) \odot X^{(t)},
\end{aligned}
\;\;\right\}
\label{eq:update}
\end{equation}
where \( X^{(t)} \) is the $t$-th step low-rank estimate, \( \mathcal{L}(\cdot) \) is a quaternion low-rank reconstruction operator, and \( \odot \) denotes the Hadamard product. At each iteration, the filled-in matrix \( C^{(t)} \) preserves the observed entries from \( M \) while updating the missing entries from \( X^{(t-1)} \), thereby ensuring data consistency across iterations.

A widely used reconstruction strategy in this context is the \emph{CUR} (cross) approximation, which decomposes \( A \in \mathbb{Q}^{m \times n} \) into three smaller factors. Specifically, representative columns and rows are selected as \( C = A_{:,J} \) and \( R = A_{I,:} \), where \( J \subset [n] \) and \( I \subset [m] \) are index sets of size \( r \), corresponding to the desired rank.  
The coupling matrix \( U \in \mathbb{Q}^{r \times r} \) links \( C \) and \( R \) via
\begin{equation*}
A \approx C U R.
\end{equation*}
The optimal \( U \), minimizing the Frobenius norm of the approximation error, is obtained from
\begin{equation*}
U = \arg\min_{U} \|A - C U R\|_F,
\end{equation*}
whose quaternion solution is
\begin{equation}
U = C^{\dagger} A R^{\dagger}.
\label{eq:Uopt}
\end{equation}
 An alternative, often termed the cross approximation, defines
\begin{equation}
U = W^{\dagger}, \quad W = A_{I,J},
\label{eq:Ucross}
\end{equation}
using the intersection submatrix \( W \).  
When the index sets \( I \) and \( J \) are selected appropriately, \eqref{eq:Ucross} and \eqref{eq:Uopt} yield equivalent results.

In practice, formulation \eqref{eq:Uopt} is more stable as it jointly incorporates row and column information via \( C \) and \( R \). However, it incurs a higher computational cost since it requires pseudoinverses of larger matrices.  
In contrast, the cross form \eqref{eq:Ucross} is computationally cheaper but its accuracy strongly depends on the selection of representative rows and columns.  
The pseudoinverse step is thus central to CUR-based completion, as each reconstruction involves computing one or more pseudoinverses of quaternion submatrices. The conventional QSVD-based pseudoinverse provides high accuracy but is computationally expensive and poorly suited for large-scale quaternion image or video data.  
To address this limitation, iterative quaternion solvers such as the QNS, and the proposed QHPI$19$, QRAPID and QSAI schemes can be employed. These algorithms require only quaternion matrix multiplications and adjoints, thereby achieving faster runtimes and lower memory usage. To enhance reconstruction quality, spatial regularization can be incorporated into the CUR framework. To further improve completion quality, CUR can be augmented with mild spatial priors. For instance, after each reconstruction step $X^{(t)}$, applying a two-dimensional Gaussian filter with standard deviation $\sigma = 0.5$ 
acts as a spatial regularizer. This suppresses noise and reduces visual artifacts, leading to consistent improvements 
in PSNR and SSIM metrics.

To evaluate performance, the proposed quaternion iterative methods were tested on the Kodak color image \textit{kodim16} (\(512 \times 768 \times 3\)), where 70\% of the pixels were randomly removed. A rank-$60$ CUR-based completion was performed for $25$ iterations and compared with the QSVD-based baseline. The reconstructed results, shown in Figure~\ref{fig:six_panel}, demonstrate that all iterative approaches achieve reconstruction quality comparable to QSVD while requiring significantly less computation time.  
The evolution of PSNR and SSIM over iterations, presented in Figure~\ref{fig:psnr_ssim_grid}, further confirms that the proposed methods yield stable and high-quality reconstructions with substantially improved efficiency.

\begin{figure*}[t]
  \centering
  \includegraphics[width=0.32\textwidth]{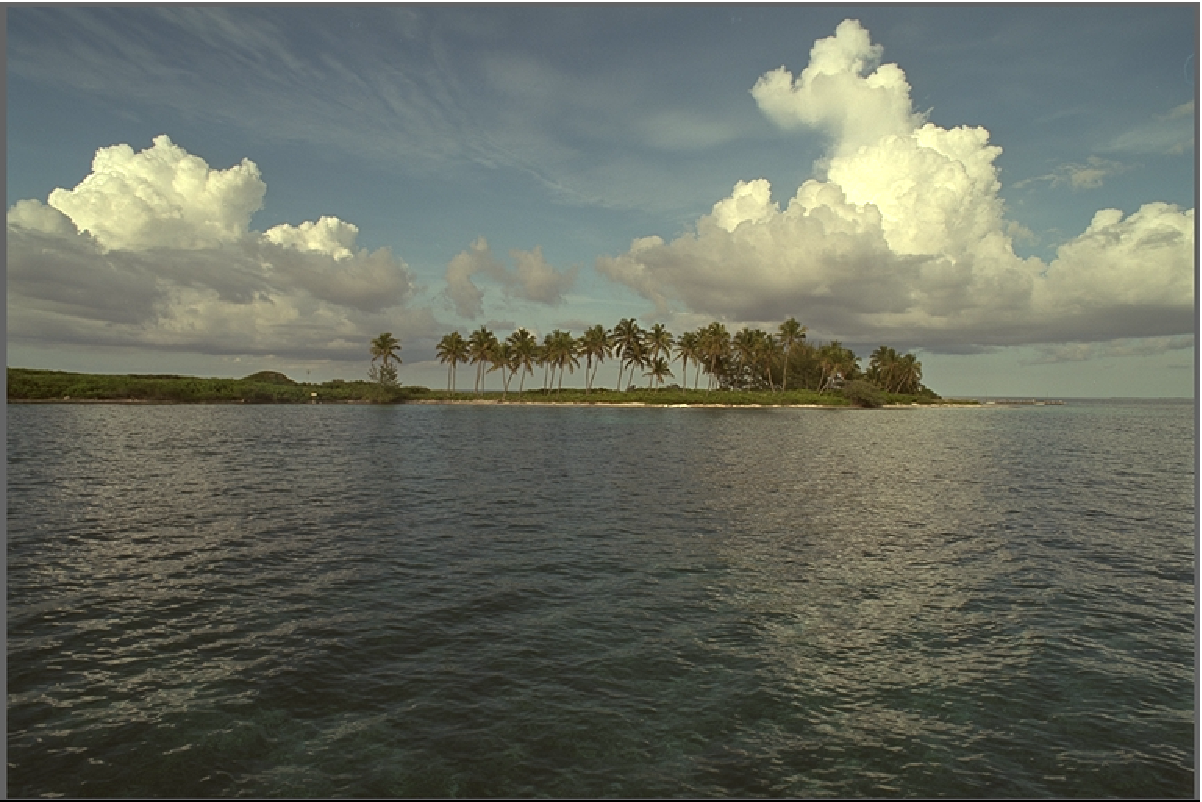}\hfill
  \includegraphics[width=0.32\textwidth]{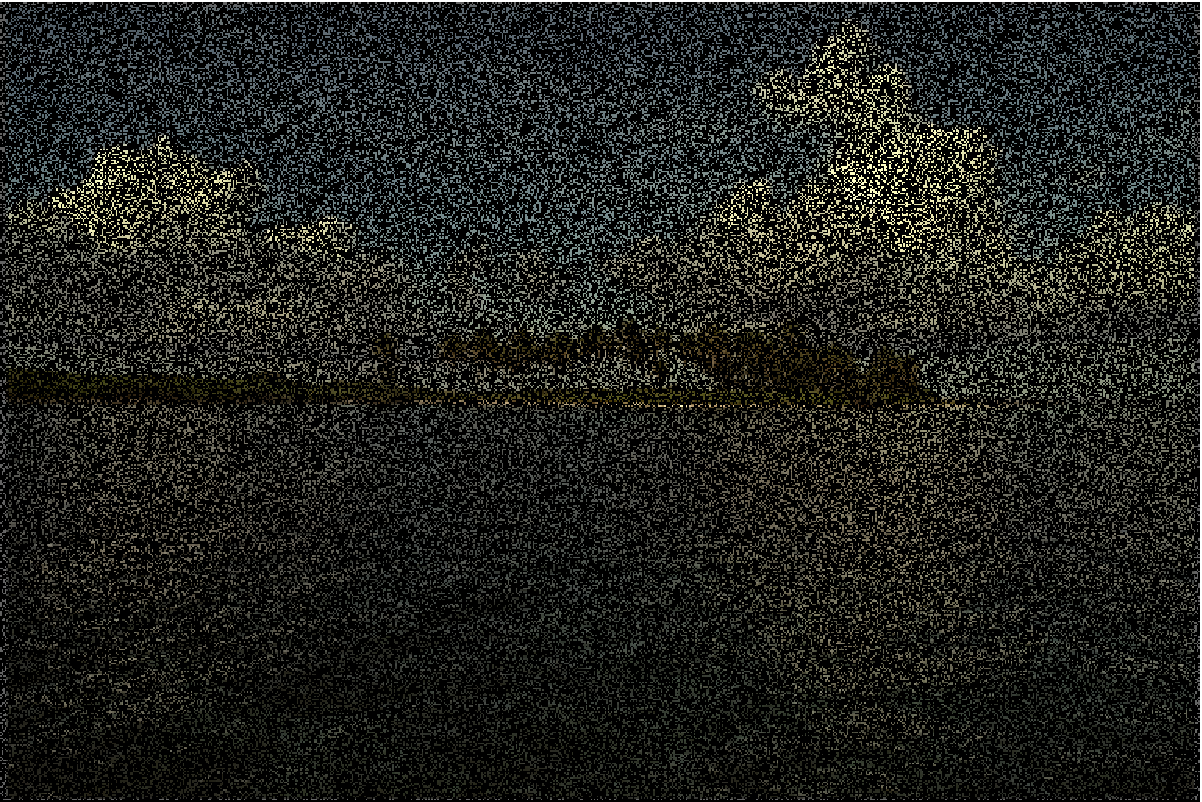}\hfill
  \includegraphics[width=0.32\textwidth]{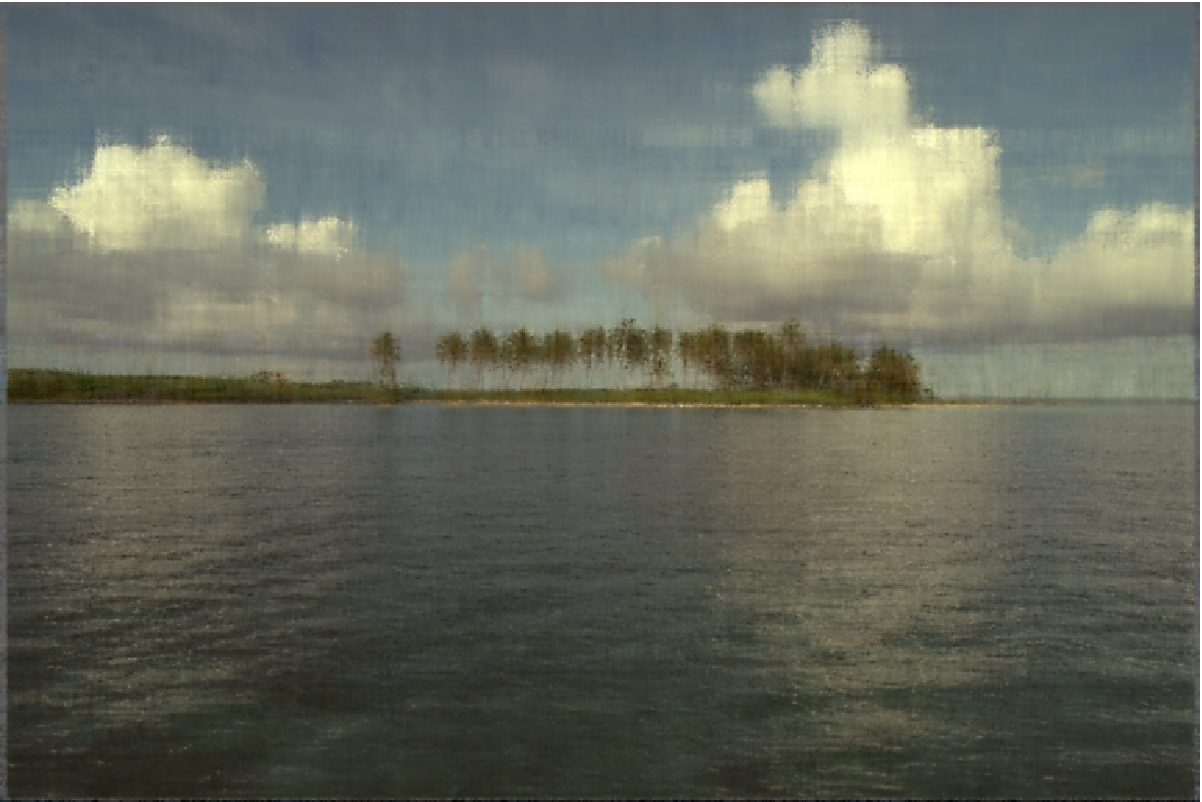}
  
  \vspace{4pt}
  
  \parbox{0.32\textwidth}{\centering  (a) Original}\hfill
  \parbox{0.32\textwidth}{\centering  (b) Masked \\  PSNR = 8.83 dB}\hfill
  \parbox{0.32\textwidth}{\centering  (c) Recovered (QSVD) \\  PSNR = 27.8005 dB \\  SSIM = 0.8111 \\  Time = 235.5673 s}
  
  \vspace{8pt}
  
  \includegraphics[width=0.32\textwidth]{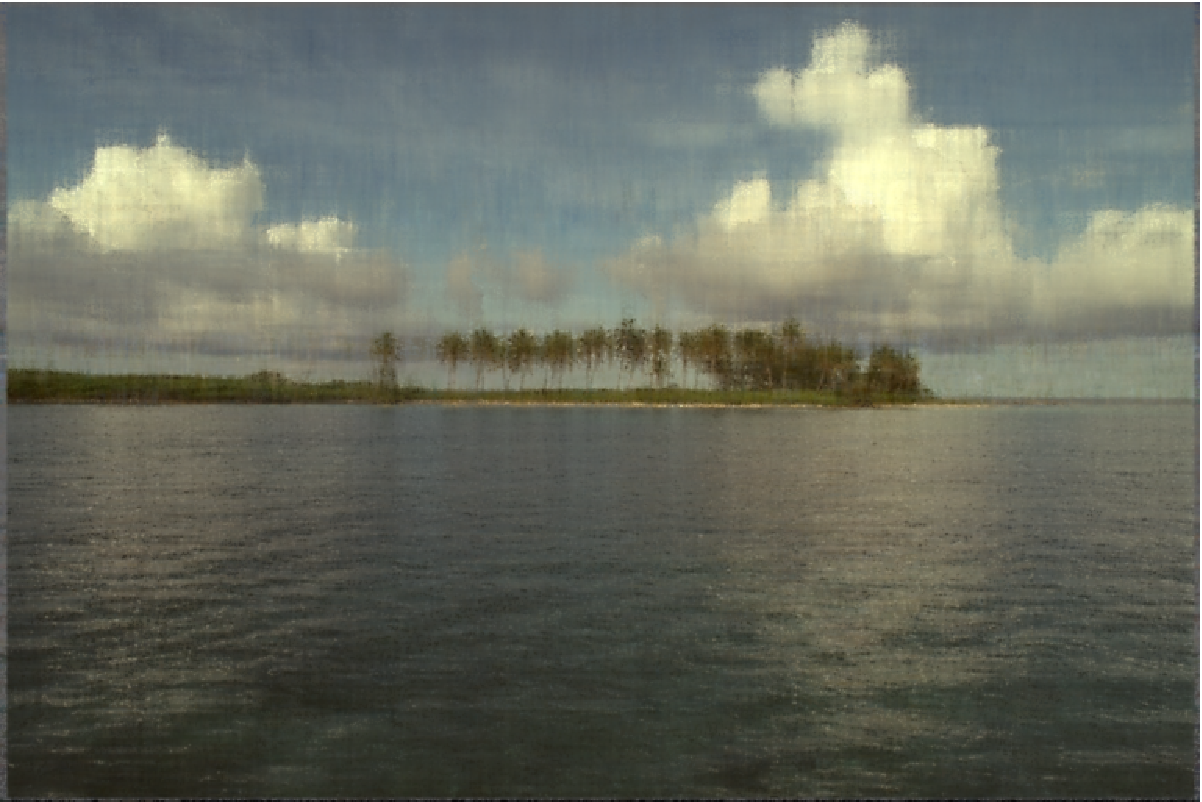}\hfill
  \includegraphics[width=0.32\textwidth]{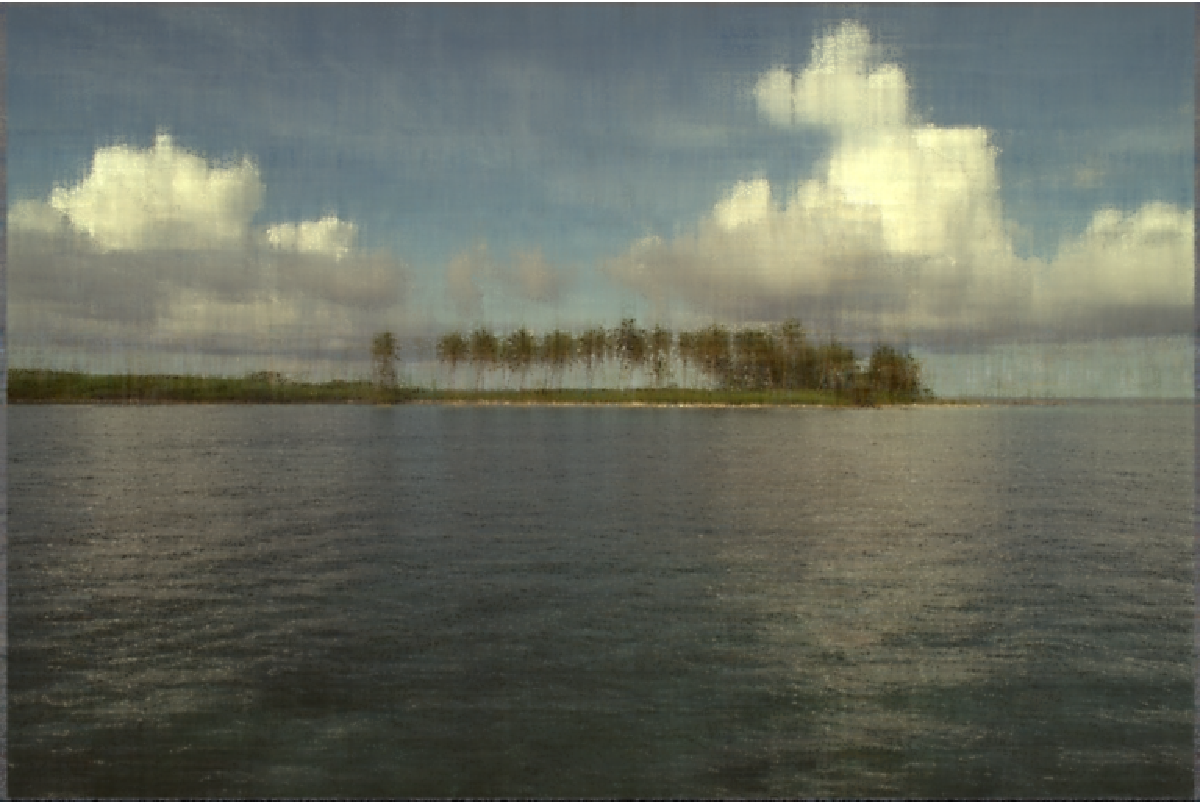}\hfill
  \includegraphics[width=0.32\textwidth]{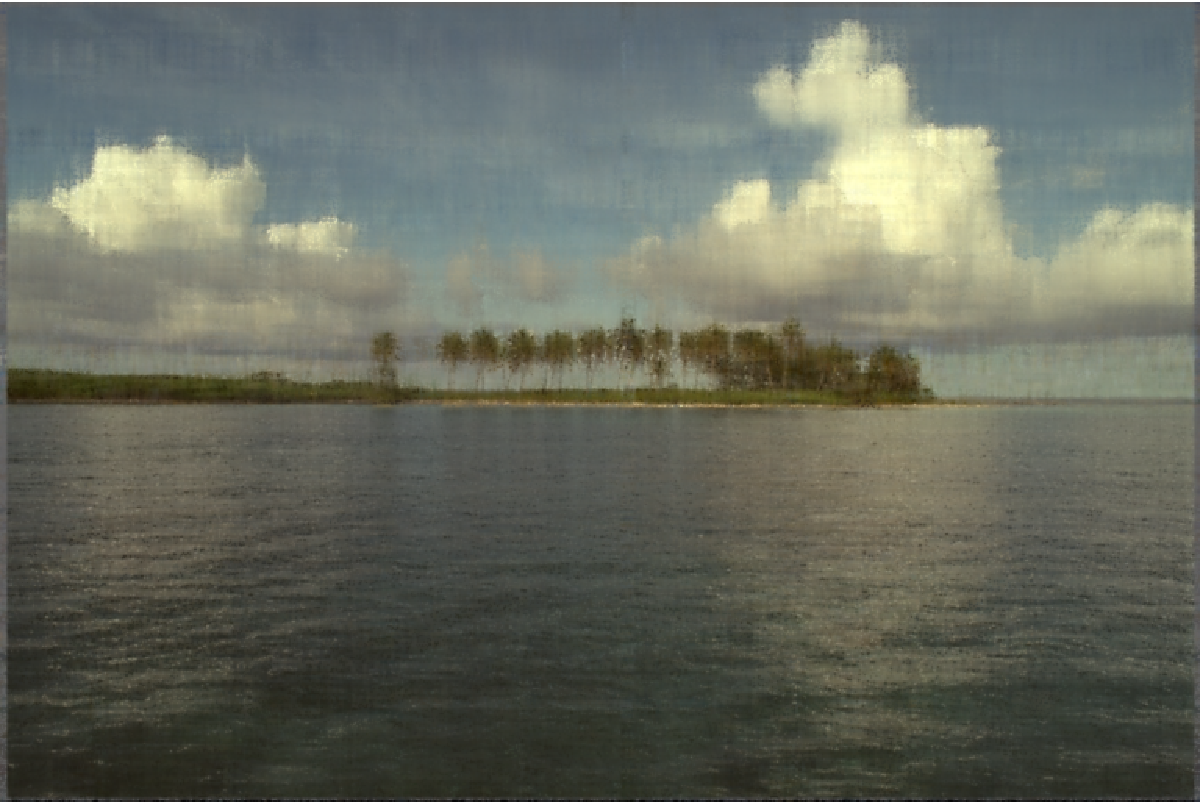}
  
  \vspace{4pt}
  
  \parbox{0.32\textwidth}{\centering  (d) Recovered (QSAI) \\  PSNR = 27.8559 dB \\  SSIM = 0.8113 \\  Time = 135.7565 s}\hfill
  \parbox{0.32\textwidth}{\centering  (e) Recovered (QRAPID) \\  PSNR = 27.8515 dB \\  SSIM = 0.8114 \\  Time = 168.2499 s}\hfill
  \parbox{0.32\textwidth}{\centering  (f) Recovered (QHPI$19$) \\  PSNR = 27.9134 dB \\  SSIM = 0.8126 \\  Time = 150.3917 s}
  
\caption{Quaternion CUR-based image completion results for the \textsc{Kodim16} image with 70\% missing pixels. 
Panels~(a)–(b) show the original and masked images, respectively, while panels~(c)–(f) present reconstructions using QSVD, QSAI, QRAPID, and QHPI$19$ -based pseudoinverses. 
For each method, the corresponding PSNR (dB), SSIM, and total runtime (seconds) are reported below the image.}
\label{fig:six_panel}
\end{figure*}

\begin{figure}[hbt!]
  \centering
  \includegraphics[width=0.48\textwidth]{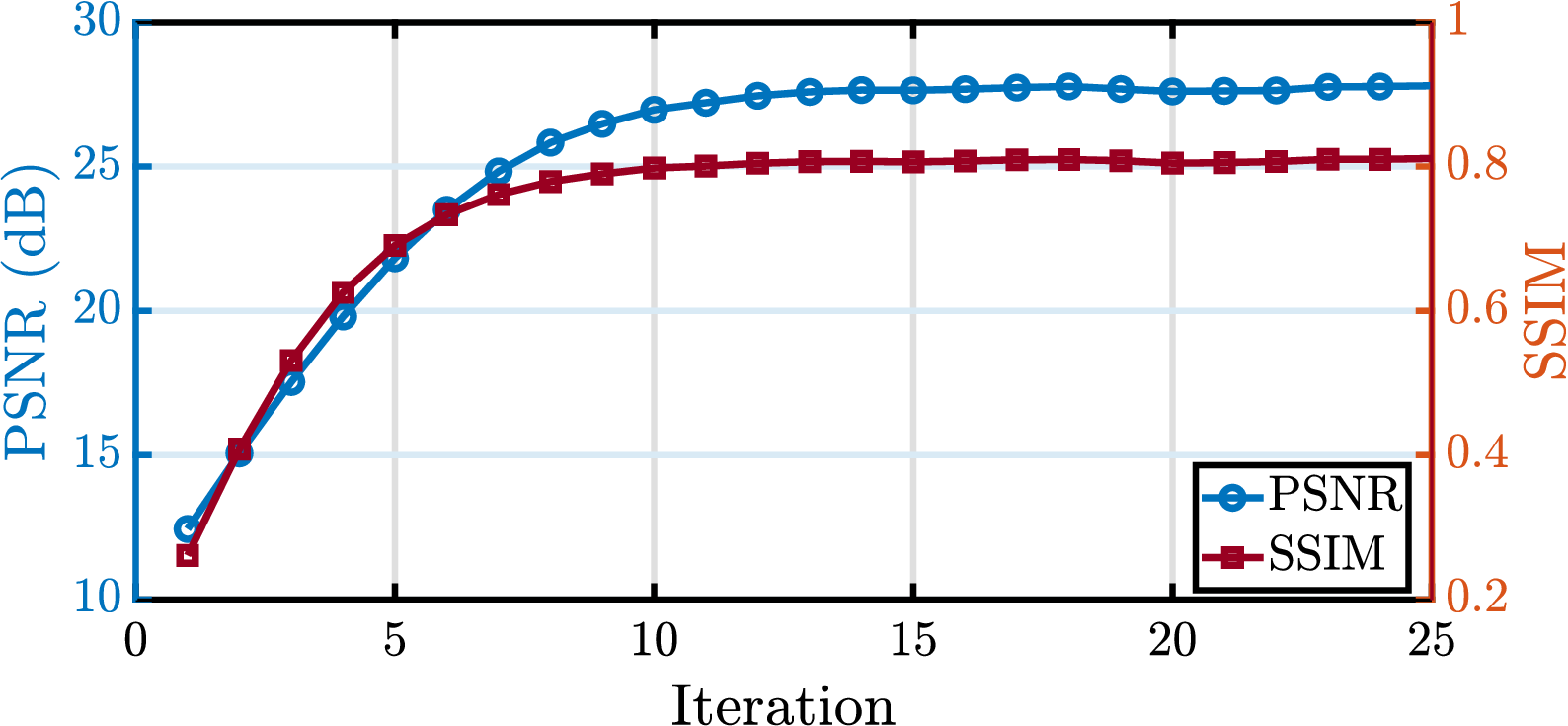}\hfill
  \includegraphics[width=0.48\textwidth]{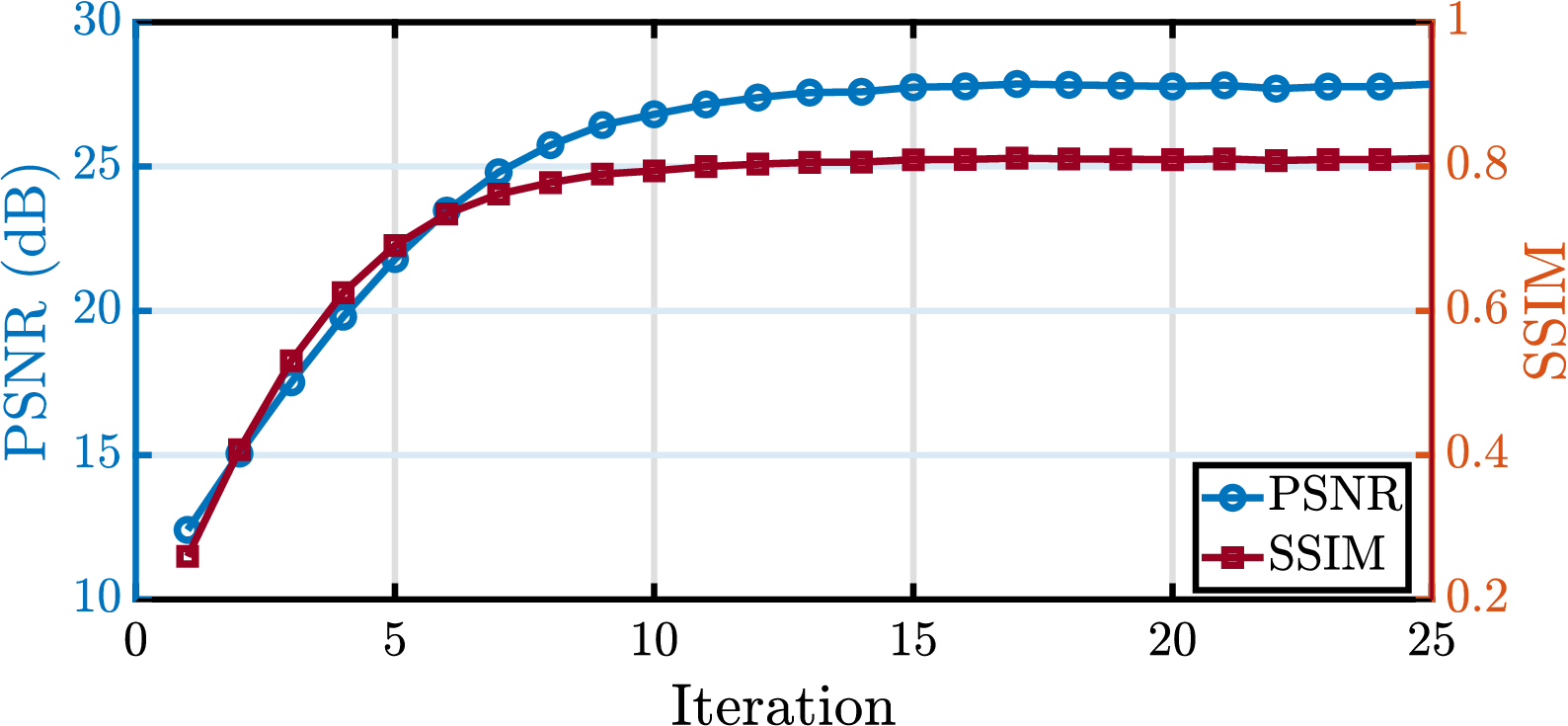}
  
  \vspace{4pt}
  
  \parbox{0.48\textwidth}{\centering  (a) QSVD: PSNR / SSIM vs iteration}\hfill
  \parbox{0.48\textwidth}{\centering  (b) QSAI: PSNR / SSIM vs iteration}
  
  \vspace{8pt}
  
  \includegraphics[width=0.48\textwidth]{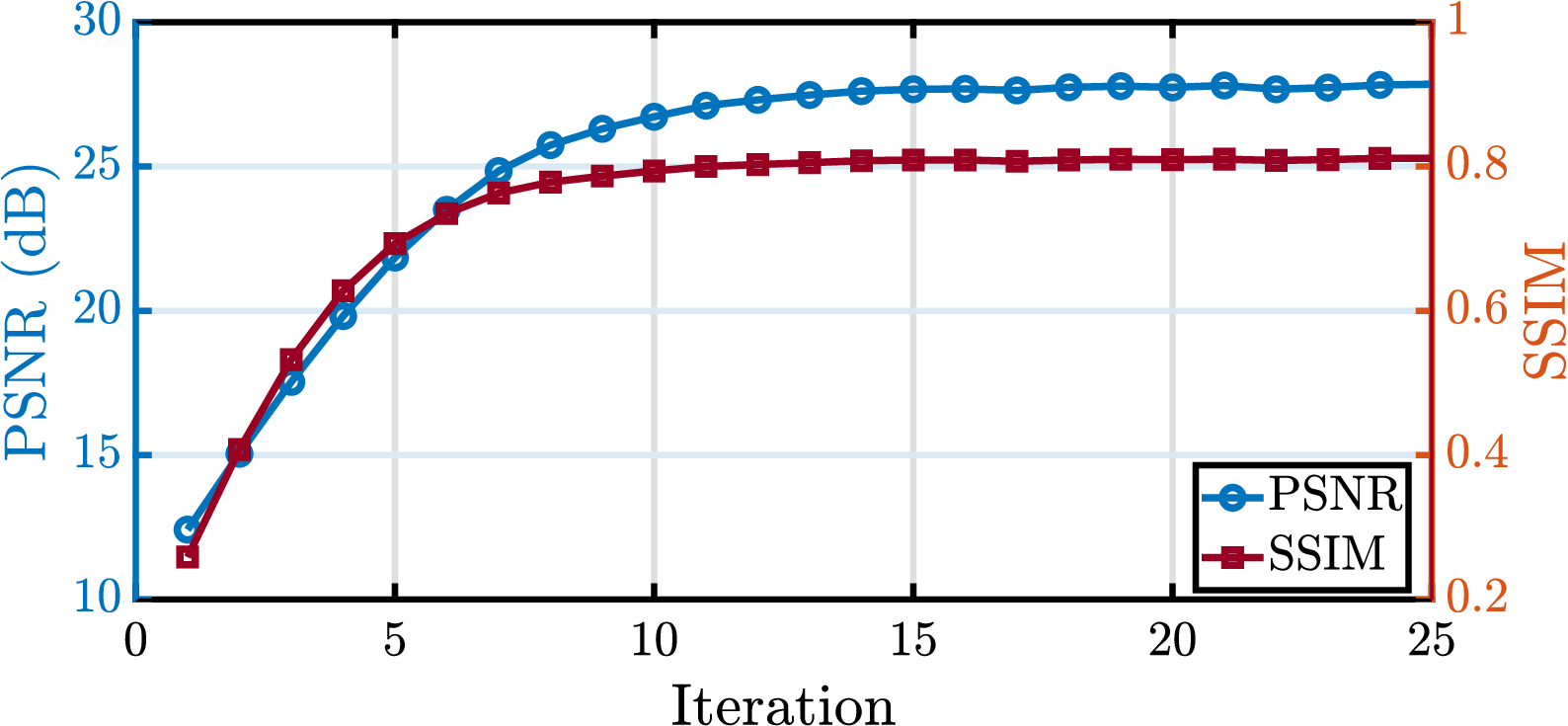}\hfill
  \includegraphics[width=0.48\textwidth]{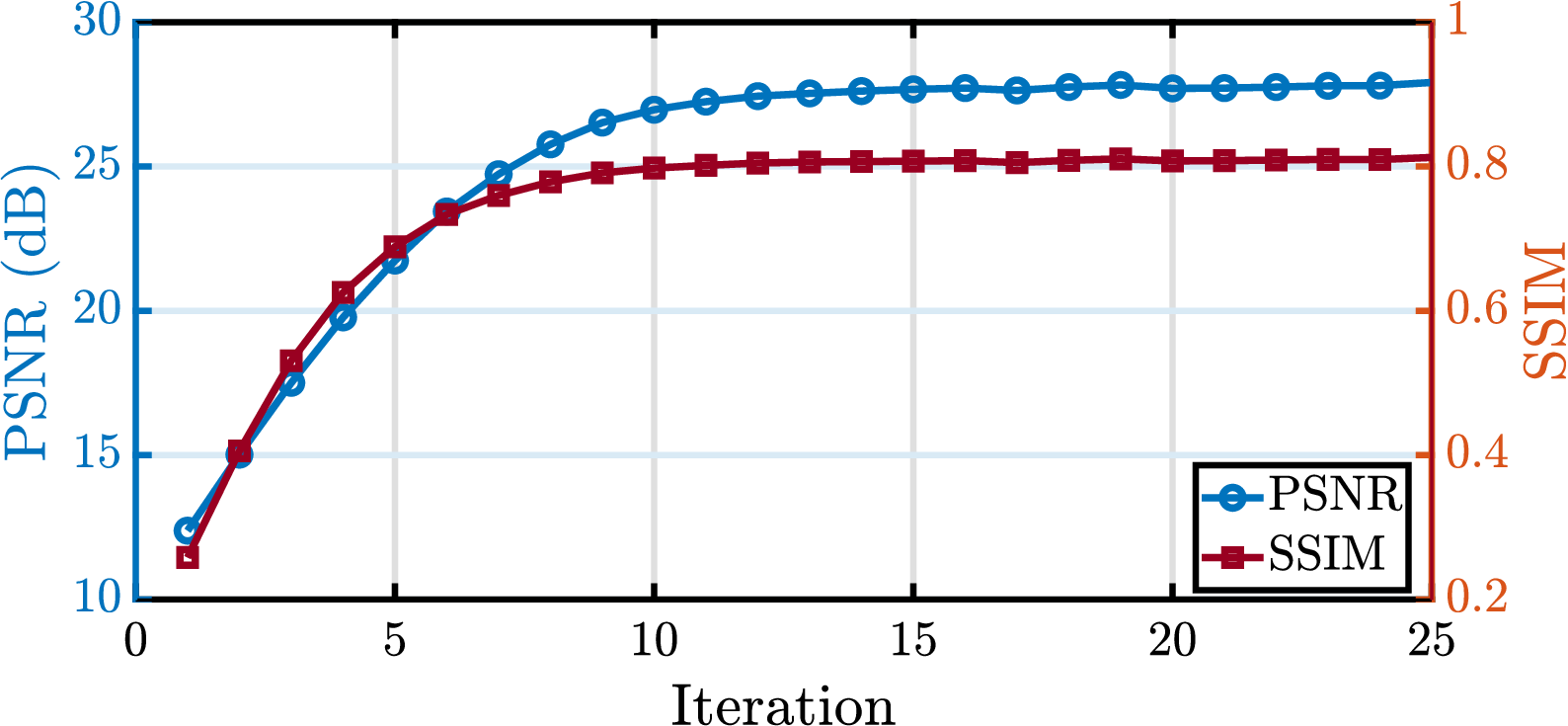}
  
  \vspace{4pt}
  
  \parbox{0.48\textwidth}{\centering  (c) QRAPID: PSNR / SSIM vs iteration}\hfill
  \parbox{0.48\textwidth}{\centering  (d) QHPI$19$: PSNR / SSIM vs iteration}
  
\caption{Evolution of reconstruction quality during quaternion CUR-based image completion of \textsc{Kodim16}. 
Each subplot shows the progression of PSNR and SSIM with respect to the number of iterations for different pseudoinverse computation methods: (a) QSVD, (b) QSAI, (c) QRAPID, and (d) QHPI$19$.}
  \label{fig:psnr_ssim_grid}
\end{figure}

\subsection{Filtering of chaotic three-dimensional signals}
Chaotic dynamical systems often produce multidimensional outputs that are highly sensitive to noise and delays, making reliable signal recovery a difficult task. To demonstrate the use of quaternion-valued filtering in such contexts, we focus on the Lorenz system, a canonical model in nonlinear dynamics known for its chaotic trajectories. The Lorenz equations are given by
\begin{equation}
\left.
\begin{array}{rcl}
\displaystyle \frac{du}{dt} &=& \sigma (v-u), \\[1ex]
\displaystyle \frac{dv}{dt} &=& u(\gamma - w) - v, \\[1ex]
\displaystyle \frac{dw}{dt} &=& uv - \delta w,
\end{array}
\right\}
\label{eq:lorenz}
\end{equation}
where $\sigma$, $\delta$, and $\gamma$ are positive constants controlling the chaotic dynamics. 

Let the solutions of Eq.\eqref{eq:lorenz} be denoted by $u(t)$, $v(t)$, and $w(t)$. These are combined into a quaternion-valued reference signal
\[
\mathbf{s}(t) = u(t)\,\mathbf{i} + v(t)\,\mathbf{j} + w(t)\,\mathbf{k}.
\]
The observed input is assumed to suffer from both delay and additive noise, and is modeled as
\[
\mathbf{x}(t) = u(t-\tau)\,\mathbf{i} + v(t-\tau)\,\mathbf{j} + w(t-\tau)\,\mathbf{k} + \mathbf{\eta}(t),
\]
where $\tau=1$ is the delay parameter, and $\mathbf{\eta}(t)$ denotes purely imaginary quaternion noise. The goal is to design a quaternion filter $\{\mathbf{h}_m\}_{m=0}^{p}$, with coefficients
\[
\mathbf{h}_m = h_{m,0} + h_{m,1}\mathbf{i} + h_{m,2}\mathbf{j} + h_{m,3}\mathbf{k},
\]
such that the filtered signal approximates the clean target signal
\begin{equation}
\mathbf{s}(t) \approx \sum_{m=0}^{p} \mathbf{x}(t-m) \mathbf{h}_m.
\label{eq:conv}
\end{equation}
Rewriting Eq.\eqref{eq:conv} in matrix–vector form gives
\[
\mathbf{X}\mathbf{h} = \mathbf{s},
\]
where
\[
\mathbf{X} = 
\begin{bmatrix}
\mathbf{x}(t) & \mathbf{x}(t-1) & \cdots & \mathbf{x}(t-p) \\
\mathbf{x}(t+1) & \mathbf{x}(t) & \cdots & \mathbf{x}(t-p+1) \\
\vdots & \vdots & \ddots & \vdots \\
\mathbf{x}(t+p) & \mathbf{x}(t+p-1) & \cdots & \mathbf{x}(t)
\end{bmatrix},
\quad
\mathbf{h} = 
\begin{bmatrix}
\mathbf{h}_0 \\ \mathbf{h}_1 \\ \vdots \\ \mathbf{h}_p
\end{bmatrix},
\quad
\mathbf{s} =
\begin{bmatrix}
\mathbf{s}(t) \\ \mathbf{s}(t+1) \\ \vdots \\ \mathbf{s}(t+p)
\end{bmatrix}.
\]
Here, $\mathbf{X}$ is a quaternion data matrix constructed from delayed input samples, $\mathbf{h}$ is the coefficient vector, and $\mathbf{s}$ is the target signal vector. 
The optimal coefficients are obtained by solving $\mathbf{h} = \mathbf{X}^{\dagger}\mathbf{s}$. To evaluate the pseudoinverse required for solving this system, we employ three different approaches: a direct pseudoinverse computation, the QNS iterative scheme~\cite{leplat2025iterative}, and proposed QSAI iterative method. The filtered output is $\hat{\mathbf{s}} = \mathbf{X}\mathbf{h}$, and the recovery error is measured as
\[
\varepsilon = \frac{\|\hat{\mathbf{s}} - \mathbf{s}\|}{\|\mathbf{s}\|}.
\]
To assess the effectiveness of quaternion-based filtering, we consider the Lorenz system with standard chaotic parameters $\sigma = 10$, $\delta = 8/3$, and $\gamma = 28$, and initial condition $[1,1,1]^T$. The system is integrated over $[0,40]$ using MATLAB’s \texttt{ode45} solver with step sizes $\Delta t \in \{0.01,0.02,0.03,0.04,0.05\}$. The clean signal components $(u,v,w)$ are combined into a quaternion reference signal, while the observed input is obtained by introducing both delay and additive quaternion noise.  
\begin{table}[hbt!]
\centering
\begin{tabular}{cccc}
\toprule
Step size $\Delta t$ & Method & Time (s) & Error $\varepsilon$ \\
\midrule
$0.01$  & QSVD            & $994.45$  & $3.72\times10^{-14}$ \\
      & QNS \cite{leplat2025iterative}  & $254.35$  & $6.43\times10^{-13}$ \\
      & Proposed QSAI   & $74.34$  & $5.35\times10^{-12}$ \\
\midrule
$0.02$  & QSVD            & $115.61$  & $2.09\times10^{-14}$ \\
      & QNS \cite{leplat2025iterative} & $37.27$  & $4.96\times10^{-13}$ \\
      & Proposed QSAI   & $10.21$  & $1.15\times10^{-14}$ \\
\midrule
$0.03$  & QSVD            & $35.34$  & $2.43\times10^{-14}$ \\
      & QNS \cite{leplat2025iterative} & $14.71$  & $5.33\times10^{-13}$ \\
      & Proposed QSAI   & $4.22$  & $2.42\times10^{-14}$ \\
\midrule
$0.04$  & QSVD            & $15.11$  & $1.31\times10^{-14}$ \\
      & QNS \cite{leplat2025iterative} & $6.19$  & $1.45\times10^{-12}$ \\
      & Proposed QSAI   & $1.70$  & $7.97\times10^{-15}$ \\
\midrule
$0.05$  & QSVD            & $6.89$  & $3.08\times10^{-14}$ \\
      & QNS \cite{leplat2025iterative} & $3.29$  & $2.91\times10^{-12}$ \\
      & Proposed QSAI   & $0.98$  & $1.23\times10^{-14}$ \\
\bottomrule
\end{tabular}
\caption{Performance comparison of pseudoinverse computation methods in quaternion filtering of the Lorenz attractor ($\tau=1$) over $[0,40]$ for various step sizes $\Delta t$.}
\label{tab:methods_comparison_dt}
\end{table}
Table~\ref{tab:methods_comparison_dt} summarizes the performance of the three pseudoinverse computation methods with varying step sizes. For each case, the CPU time and the relative reconstruction error $\varepsilon$ are reported. 
The proposed QSAI method achieves reconstruction errors comparable to QSVD and QNS while reducing computation time. This confirms that the QSAI approach is not only accurate, but also well suited for real-time quaternion signal filtering.

For visualization, Figure~\ref{fig:lorenz_filtering_all} illustrates the case $\Delta t = 0.01$. 
The first row shows the Lorenz attractor and the corresponding clean signal components $(u,v,w)$. 
The second row depicts the noisy, delayed input and the reconstruction obtained via QSVD, while the third row compares the iterative QNS and QSAI results. 
Both iterative methods successfully recover the chaotic trajectories, but QSAI achieves comparable accuracy with a markedly lower computational cost.
\begin{figure}[htbp!]
  \centering
  \includegraphics[width=0.48\textwidth]{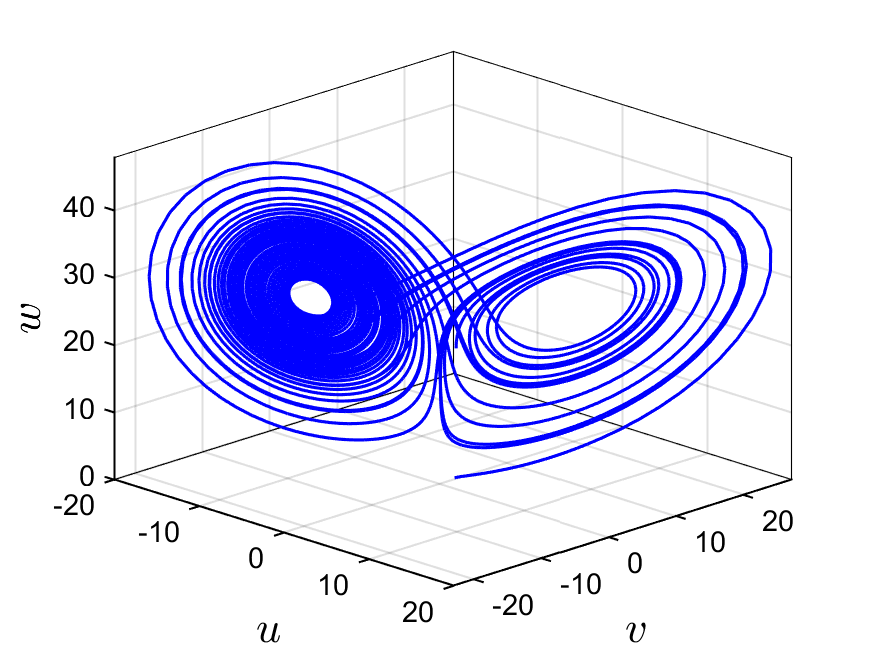}\hfill
  \includegraphics[width=0.48\textwidth]{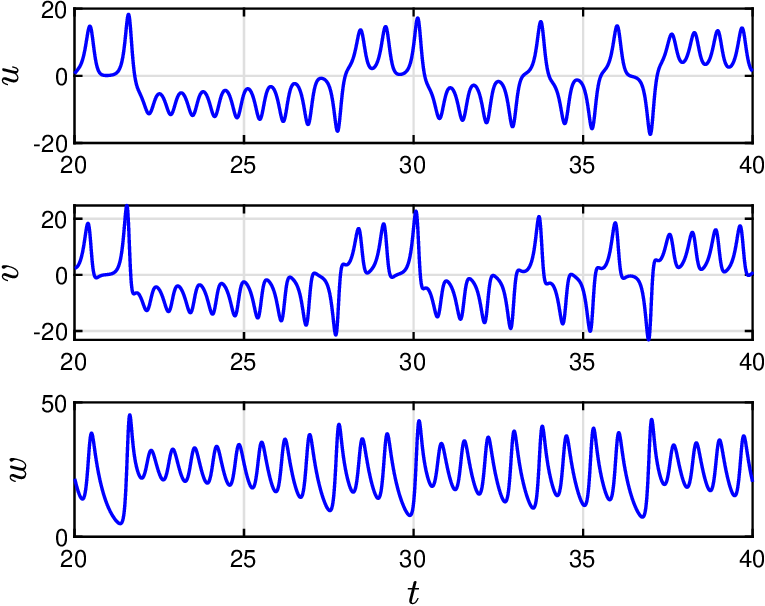}
  
 \vspace{4pt}
 
  \parbox{0.48\textwidth}{\centering  (a) Three-dimensional Lorenz attractor trajectory}\hfill
  \parbox{0.48\textwidth}{\centering  (b) Clean reference signal components $(u,v,w)$}
  
  \vspace{4pt}
  
  \includegraphics[width=0.48\textwidth]{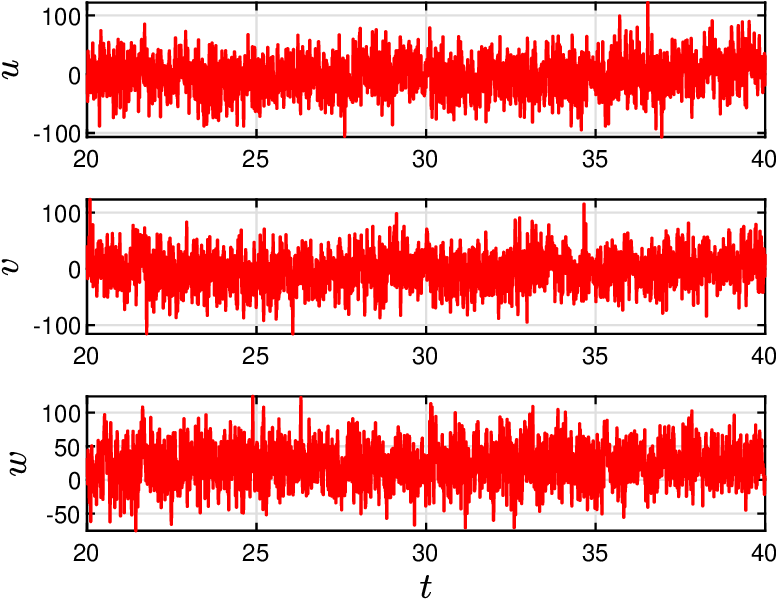}\hfill
  \includegraphics[width=0.48\textwidth]{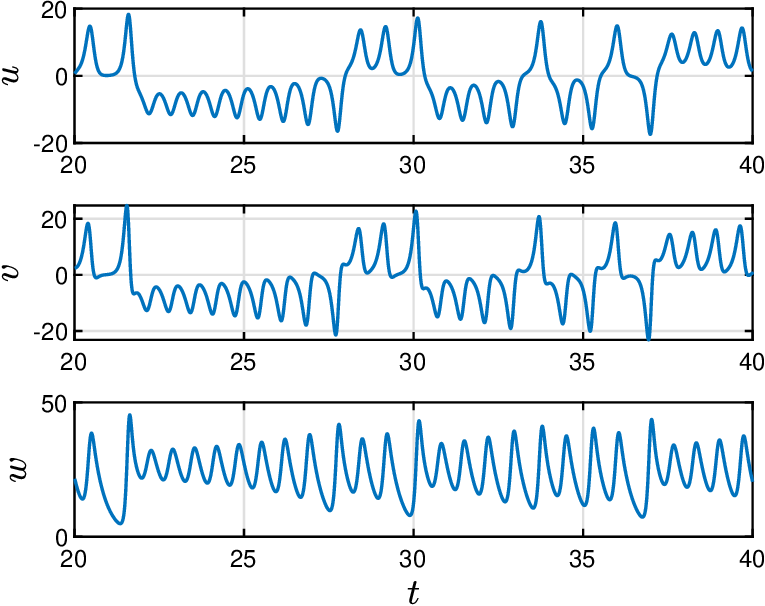}
  
  \vspace{4pt}
  
  \parbox{0.48\textwidth}{\centering  (c) Observed input: delayed and noisy signal components}\hfill
  \parbox{0.48\textwidth}{\centering  (d) Reconstructed signal components using QSVD}
  
  \vspace{4pt}
  
  \includegraphics[width=0.48\textwidth]{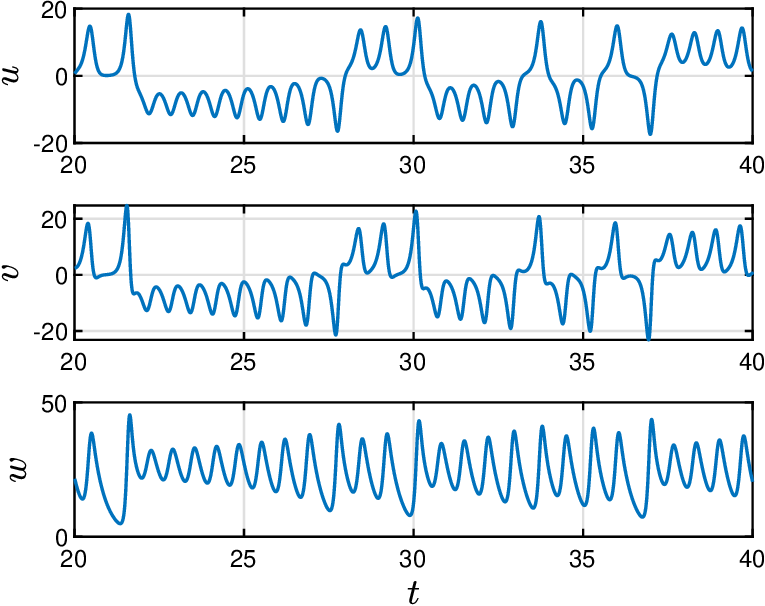}\hfill
  \includegraphics[width=0.48\textwidth]{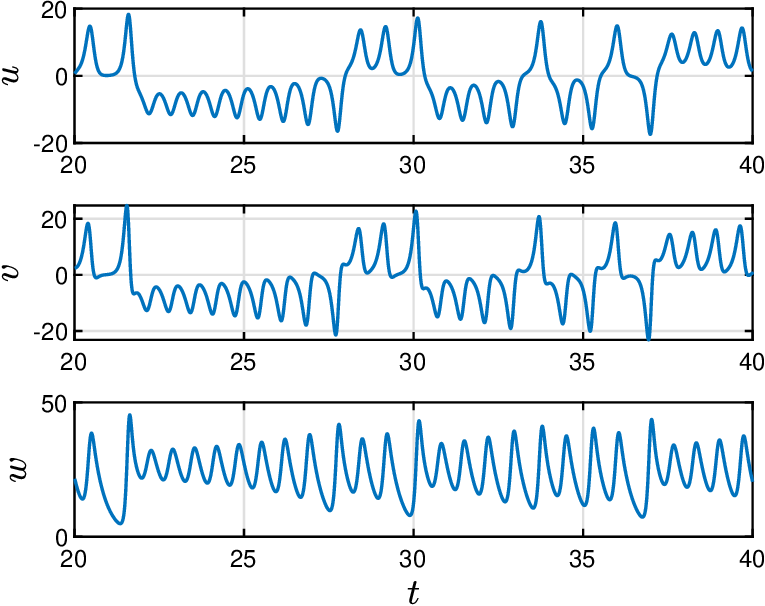}
  
  \vspace{4pt}
 
  \parbox{0.48\textwidth}{\centering  (e) Reconstructed signal components using QNS method}\hfill
  \parbox{0.48\textwidth}{\centering  (f) Reconstructed signal components using QSAI}
    \caption{Quaternion filtering of chaotic Lorenz signals: comparison of reference, noisy input, and reconstructed outputs using different pseudoinverse computation schemes.}
  \label{fig:lorenz_filtering_all}
\end{figure}

\section{Conclusions}\label{sec6}
This paper presented three quaternion-based iterative methods—QRAPID, QSAI, and QHPI$19$—for efficiently computing the Moore–Penrose inverse of quaternion matrices. Rigorous analyses of convergence and numerical stability were carried out, ensuring the reliability and robustness of the proposed schemes. Extensive numerical experiments demonstrated that these methods achieve accuracy comparable to QSVD and QNS approaches while requiring significantly less computational time, thereby offering practical advantages for large-scale quaternion inverse problems. In addition to standalone inversion, the QSAI method was successfully employed as a preconditioner for quaternion Krylov subspace solvers, substantially accelerating convergence in large and ill-conditioned quaternion linear systems. The effectiveness of the proposed framework was further validated through two representative applications: image inpainting and the filtering of chaotic signals, thereby confirming both accuracy and computational efficiency.

Future work will focus on extending these algorithms to broader algebraic settings, such as quaternion tensors, split and dual quaternions, and block-structured Moore–Penrose inverses. Moreover, GPU-based acceleration and mixed-precision implementations present promising avenues for enhancing scalability, as the proposed methods are matrix-free and primarily rely on matrix–matrix multiplications.

\section*{Funding}
\begin{itemize}
\item Ratikanta Behera is supported by the Anusandhan National Research Foundation (ANRF), Government of India, under Grant No. EEQ/2022/001065.
\end{itemize}

\section*{Conflict of Interest}
The authors would like to assure the readers that they have no potential conflicts of interest to report.
\section*{Data Availability}
In the context of this article, it is important to clarify that the data sets created or examined during the course of this study can be shared on request.

\section*{ORCID}
Biswarup Karmakar~\orcidB \href{https://orcid.org/0009-0003-5635-5425}{ \hspace{2mm}\textcolor{lightblue}{https://orcid.org/0009-0003-5635-5425}} \\
Neha Bhadala \orcidC \href{https://orcid.org/0009-0001-9249-0611}{ \hspace{2mm}\textcolor{lightblue}{https://orcid.org/0009-0001-9249-0611}} \\
Ratikanta Behera\orcidA \href{https://orcid.org/0000-0002-6237-5700}{ \hspace{2mm}\textcolor{lightblue}{ https://orcid.org/0000-0002-6237-5700}}\\

\bibliography{References}
\end{document}